\documentclass{article}
\usepackage[english]{babel}

\usepackage{amsmath,enumerate,hyperref}

\usepackage{amssymb}
\usepackage{listings}
\usepackage[table,dvipsnames]{xcolor} 

\usepackage{amsthm}
\newtheorem{lemma}{Lemma}
\newtheorem{proposition}{Proposition}

\newtheorem{remark}{Remark}
\newtheorem{theorem}{Theorem}
\newtheorem{definition}{Definition}
\usepackage{a4wide}

\definecolor{codegreen}{rgb}{0,0.6,0}
\definecolor{codegray}{rgb}{0.5,0.5,0.5}
\definecolor{codepurple}{rgb}{0.58,0,0.82}
\definecolor{backcolour}{rgb}{0.95,0.95,0.92}
\lstdefinestyle{mystyle}{
    backgroundcolor=\color{backcolour},   
    commentstyle=\color{codegreen},
    keywordstyle=\color{magenta},
    numberstyle=\tiny\color{codegray},
    stringstyle=\color{codepurple},
    basicstyle=\ttfamily\footnotesize,
    breakatwhitespace=false,         
    breaklines=true,                 
    captionpos=b,                    
    keepspaces=true,                 
    numbers=left,                    
    numbersep=5pt,                  
    showspaces=false,                
    showstringspaces=false,
    showtabs=false,                  
    tabsize=2
}
\usepackage{tikz}
\usetikzlibrary{patterns}
\usepackage{pgfplots}
\pgfplotsset{compat=newest}
\usetikzlibrary{calc}
\newcommand{\logLogSlopeTriangle}[5]
{

    \pgfplotsextra
    {
        \pgfkeysgetvalue{/pgfplots/xmin}{\xmin}
        \pgfkeysgetvalue{/pgfplots/xmax}{\xmax}
        \pgfkeysgetvalue{/pgfplots/ymin}{\ymin}
        \pgfkeysgetvalue{/pgfplots/ymax}{\ymax}

        \pgfmathsetmacro{\xArel}{#1}
        \pgfmathsetmacro{\yArel}{#3}
        \pgfmathsetmacro{\xBrel}{#1-#2}
        \pgfmathsetmacro{\yBrel}{\yArel}
        \pgfmathsetmacro{\xCrel}{\xArel}

        \pgfmathsetmacro{\lnxB}{\xmin*(1-(#1-#2))+\xmax*(#1-#2)} 
        \pgfmathsetmacro{\lnxA}{\xmin*(1-#1)+\xmax*#1} 
        \pgfmathsetmacro{\lnyA}{\ymin*(1-#3)+\ymax*#3} 
        \pgfmathsetmacro{\lnyC}{\lnyA+#4*(\lnxA-\lnxB)}
        \pgfmathsetmacro{\yCrel}{\lnyC-\ymin)/(\ymax-\ymin)} 
        
        \coordinate (A) at (rel axis cs:\xArel,\yArel);
        \coordinate (B) at (rel axis cs:\xBrel,\yBrel);
        \coordinate (C) at (rel axis cs:\xCrel,\yCrel);

        \draw[#5]   (A)-- node[pos=0.5,anchor=north] {1}
                    (B)-- 
                    (C)-- node[pos=0.5,anchor=west] {#4}
                    cycle;
    }
}
\newcommand{\logLogSlopeTriangleinv}[5]
{
    \pgfplotsextra
    {
        \pgfkeysgetvalue{/pgfplots/xmin}{\xmin}
        \pgfkeysgetvalue{/pgfplots/xmax}{\xmax}
        \pgfkeysgetvalue{/pgfplots/ymin}{\ymin}
        \pgfkeysgetvalue{/pgfplots/ymax}{\ymax}

        \pgfmathsetmacro{\xArel}{#1}
        \pgfmathsetmacro{\yArel}{#3}
        \pgfmathsetmacro{\xBrel}{#1-#2}
        \pgfmathsetmacro{\yBrel}{\yArel}
        \pgfmathsetmacro{\xCrel}{\xBrel}

        \pgfmathsetmacro{\lnxB}{\xmin*(1-(#1-#2))+\xmax*(#1-#2)} 
        \pgfmathsetmacro{\lnxA}{\xmin*(1-#1)+\xmax*#1} 
        \pgfmathsetmacro{\lnyA}{\ymin*(1-#3)+\ymax*#3} 
        \pgfmathsetmacro{\lnyC}{\lnyA+#4*(\lnxA-\lnxB)}
        \pgfmathsetmacro{\yCrel}{(\lnyC-\ymin)/(\ymax-\ymin)} 
        
        \coordinate (A) at (rel axis cs:\xArel,\yArel);
        \coordinate (B) at (rel axis cs:\xBrel,\yBrel);
        \coordinate (C) at (rel axis cs:\xCrel,\yCrel);

        \draw[#5]   (A)-- node[pos=0.5,anchor=north] {1}
                    (B)-- node[pos=0.5,anchor=east] {#4}
                    (C)-- 
                    cycle;
    }
}

\title{$\varphi-$FD : A well-conditioned finite difference method inspired by $\varphi-$FEM for general geometries on elliptic PDEs\footnote{
This work was supported by the Agence Nationale de la Recherche, Project PhiFEM, under grant ANR-22- CE46-0003-01.}}
\author{Michel Duprez\footnote{MIMESIS team, Inria de l'Université de Lorraine, MLMS team, Universit\'e de Strasbourg, 2 Rue Marie Hamm, 67000 Strasbourg, France, 
\texttt{michel.duprez@inria.fr}}, 
Vanessa Lleras\footnote{IMAG, Univ Montpellier, CNRS UMR 5149, 499-554 Rue du Truel, 34090 Montpellier, France, \texttt{vanessa.lleras@umontpellier.fr}},
Alexei Lozinski\footnote{Université de Franche-Comté, Laboratoire de mathématiques de Besançon, UMR~CNRS~6623, 16 route de Gray, 25030 Besançon Cedex, France, \texttt{alexei.lozinski@univ-fcomte.fr}},
Vincent Vigon\footnote{Institut de Recherche Mathématique Avancée, UMR 7501, Université de Strasbourg et CNRS, Tonus team, Inria de l'Université de Lorraine, 7 rue René Descartes, 67000 Strasbourg, France, 
\texttt{vincent.vigon@math.unistra.fr}} 
and Killian Vuillemot\footnote{MIMESIS team, Inria de l'Université de Lorraine, MLMS team, Universit\'e de Strasbourg, 2 Rue Marie Hamm, 67000 Strasbourg, France, 
IMAG, Univ Montpellier, CNRS UMR 5149, 499-554 Rue du Truel, 34090 Montpellier, France, \texttt{killian.vuillemot@umontpellier.fr}}
}

\begin{document}
\maketitle

\begin{abstract}
    This paper presents a new finite difference method, called $\varphi$-FD, inspired by the $\varphi$-FEM {approach} for solving elliptic partial differential equations (PDEs) on general geometries. The proposed method uses Cartesian grids, ensuring   simplicity in implementation. Moreover, contrary to the previous finite {difference} scheme on non-rectangular domain, the associated matrix is well-{conditioned}.
    The use of a level-set function for the geometry description makes this approach relatively flexible. 
    We prove the quasi-optimal convergence rates in several norms and the fact {that} the matrix is well-conditioned. 
    Additionally, the paper explores the use of multigrid techniques to further accelerate the computation. 
    Finally, numerical experiments in both 2D and 3D validate the performance of the $\varphi$-FD method compared to standard finite element methods and the Shortley-Weller approach.
\end{abstract}

\section{Introduction}

We consider here the Poisson problem with homogeneous Dirichlet boundary conditions
\begin{equation}\label{eq:poisson}
    - \Delta u = f \  \text{in} \  \Omega, \quad u = 0 \ \text{on} \ \partial \Omega,
\end{equation} 
where $f\in \mathcal{C}^{0}(\Omega)$ and $\Omega \subset \mathbb R^n$ ($n=1,\cdots,3$) is a connected domain.
In the present article, we will propose a new scheme to approximate the solution to \eqref{eq:poisson} on a Cartesian grid for general geometries $\Omega$.

\paragraph{General advantages of Cartesian grids}
It is difficult and time-consuming to generate a body-fitting grid of a complex domain.
This problem can be overcome by embedding the domain in a Cartesian grid, with the following advantages: 
\begin{itemize}
\item  The grid generation is simple and fast.
\item  Boundaries or interfaces can be easily represented by level-set functions. 
\item  Computations can be parallelized. 
\item Once the problem is posed on the Cartesian grid (which is done analytically), we no longer need interpolation, except for some approaches as the multigrid one presented in \ref{sec:multi}.
\end{itemize}

\paragraph{Finite difference method}
{The use of Cartesian grids is mandatory to solve elliptic partial differential equations with finite difference approaches. To do so on complex geometries, the main approach used in the literature is the method introduced by Shortley and Weller in \cite{shortley1938numerical}.}
In \cite{Weynans2017,bramble}, the authors have developed convergence study techniques for such finite difference methods. These papers use discrete Green functions and a discrete maximum principle to obtain precise estimates of the coefficients of the inverse matrix. These estimates sometimes lead to a phenomenon of supraconvergence \cite{yoon}, which means that the numerical scheme converges to a higher order than the one expected.
In \cite{Weynans_2012}, they have considered elliptic problems with immersed interfaces.
It has been proposed in \cite{gibou} a second-order accurate scheme to {solve} the Poisson equation with Dirichlet boundary conditions on irregular domains.
The immersed interface method \cite{Li} is based on a Cartesian grid and is
 associated with a second-order finite difference
scheme for very general second-order elliptic and parabolic linear PDEs. They solve boundary value problems, extending past the boundary to a computational box.
So, the combination of finite difference techniques and an accurate unfitted method is, therefore, a natural idea. The drawback of these finite difference methods is that the associated matrix is not well conditioned.

\paragraph{Finite element method}
Now, let's review the techniques based on the non-conform finite elements. Initial approaches like \cite{IBMrev,glowinski,girault} have a lack of precision due to their simple treatment of boundary conditions and also produce not well-conditioned matrix. 
Over the past two decades, have emerged more accurate 
methods, including XFEM \cite{moes,haslin}, CutFEM \cite{burman1,burman2,cutfemrev}, and the Shifted Boundary Method (SBM) \cite{SBM}. 
They are mainly optimally convergent and the associated matrix is well-conditioned, but require non-standard quadrature rules or extrapolations to assemble the matrices.
To avoid these constraints, the authors of \cite{duprez} have developed a non-conforming method called $\varphi$-FEM, which uses a level-set function to describe the domain. 
$\varphi$-FEM has already been demonstrated to be faster and more accurate than the classical finite element method on several problems \cite{cotin2023phi,phiFEM2,duprez2023phi,duprez2023phiheat}. In a recent paper \cite{duprez2024varphi}, it has been proposed a combination with a machine learning approach called $\varphi$-FEM-FNO based on the Fourier Neural Operator, which needs Cartesian grids to perform discrete fast Fourier transform. 
\bigskip

{
In the present article, we propose a finite difference scheme on a Cartesian grid inspired by $\varphi$\nobreakdash-FEM.
As in this approach, the domain is described by a level-set function $\varphi$, which will be used to impose the boundary conditions by penalization. This method, which we called $\varphi$-FD, combines optimal accuracy, well conditioning of the associated matrix, and simplicity of implementation (few lines of python code with the help of \texttt{scipy} {\cite{scipy}}, see appendix).} 

\paragraph{Article outline}
The paper is organized as follows: 
Section 2 describes the expected formulation of the $\varphi$-FD method for Poisson equation with homogeneous Dirichlet boundary conditions and gives theorems on the convergence and on the conditioning of the associated matrix. 
Section \ref{sec:link} explains the parallel with the original $\varphi$-FEM method. 
    Section 4 contains the proof of the two main theorems of section 2. 
    Section 5 proposes an alternative scheme which is numerically optimally convergent. 
    Section 6  is devoted to the numerical illustration of the method and a combination of our scheme with a multigrid approach. In the appendix, we give {an example of implementation for $\varphi$-FD}  in the \texttt{python} language.

\section{Main results}\label{sec:main}

{The domain is described by a level-set function $\varphi$ such that 
\begin{equation}\label{eq:phi}
     \Omega = \{ \varphi < 0 \}\,. 
\end{equation}
In particular, its boundary $\Gamma $ is given by $ \{ \varphi = 0 \}$.
}

{
We suppose that $\Omega$ is included in $\mathcal{O}:= \prod\limits_{i=1}^n[a_i,b_i]$ with $b_i-a_i=b_j-a_j$ for $i\neq j$. 
Let $N\in\mathbb{N}^*$, $h=(b_1-a_1)/N$ and we consider the Cartesian grid covering this rectangle:
$$
\mathcal{O}_h:=\{x_{\alpha}:\alpha \in \{0,\cdots, N\}^n\}
$$ 
with $x_{\alpha_i}=a_i+\alpha_ih$ for $\alpha=(\alpha_1,\cdots,\alpha_n)$. }

{
We denote by
$$D=\begin{cases}
\{1\},&\text{ if }n=1,\\
\{(1,0),(0,1)\},&\text{ if }n=2,\\
\{(1,0,0),(0,1,0),(0,0,1)\},&\text{ if }n=3.
\end{cases}$$

We define 
the following sub-grids:
$$\Omega_h = \{x_{\alpha}\in \mathcal{O}_h:x_{\alpha}\in \Omega \mbox{ or }  x_{\alpha\pm d}\in \Omega, ~d\in D\}\,, $$
$$\Omega_h^{\text{int}} = \{x_{\alpha}\in \mathcal{O}_h:x_{\alpha}\in \Omega\}\,{.} $$
{Moreover, let $\overline{\Omega}_h$,  be the union of}
 squares with vertices $x_{\alpha}\in \mathcal{O}_h$ intersecting $\Omega$
and  $\overline{\Omega}_h^{\text{int}}$ be the union of squares with vertices $x_{\alpha}\in \mathcal{O}_h$ 
included in $\Omega$.
}
An example is given in Figure \ref{fig:situation}.

Let us now describe our finite difference method. 
We propose here a description of the scheme for any dimension, but it will be given in the two-dimensional case with explicit indices in Section \ref{sec:prove}. 
{Find a discrete function $u_h=(u_{\alpha})_{\alpha:x_{\alpha}\in\Omega_h}$ defined on $\Omega_h$ such that}
\begin{equation}\label{eq:diff}
a_h(u_h,v_h)=l_h(v_h),
\end{equation}
{for all discrete function $v_h=(v_{\alpha})_{\alpha:x_{\alpha}\in\Omega_h}$ defined on $\Omega_h$,} 
where
$$
a_h(u_h,v_h)=(-\Delta_h u_h,v_h)+b_h(u_h,v_h)+j_h(u_h,v_h)\,,
$$
and 
$$
l_h(v_h)=\sum_{\scriptstyle \alpha:x_{\alpha}\in \Omega_h^{\text{int}}}\sum_{\scriptstyle d\in D} f_{\alpha}v_{\alpha}\,,
$$
with
$f_{\alpha}=(f(x_{\alpha}))_{\alpha}$, 
    the discrete Laplacian:
\begin{equation*}
(-\Delta_h u_h,v_h)=\sum_{\scriptstyle \alpha:x_{\alpha}\in \Omega_h^{\text{int}}}\sum_{\scriptstyle d\in D}\frac{ - u_{\alpha -d} +2 u_{\alpha}- u_{\alpha +d}}{h^2}v_{\alpha},
  \end{equation*}
a penalization for the boundary conditions
\begin{equation*}
b_h(u_h,v_h) =\frac{\gamma}{2 h^2}\sum_{
(\alpha,d)\in B} 
\frac{1}{\varphi_{\alpha}^2 + \varphi_{\alpha+d}^2}
(\varphi_{\alpha+d} u_{\alpha} - \varphi_{\alpha} u_{\alpha+d}) (\varphi_{\alpha+d}
  v_{\alpha} - \varphi_{\alpha} v_{\alpha+d})
  \end{equation*} 
 and a stabilization term near the boundary
\begin{equation*}
j_h(u_h,v_h)= \sigma \sum_{(\alpha,d)\in J}\frac{- u_{\alpha-d} + 2 u_{\alpha} - u_{\alpha+d}}{h }
  \times \frac{- v_{\alpha-d} + 2 v_{\alpha} - v_{\alpha+d}}{h }
\end{equation*}
  with $\gamma,~\sigma>0$ and
  $$B=\{(\alpha,d)|\text{ the edge }x_{\alpha} - x_{\alpha+d}\text{   intersects }\Gamma\text{ and not included in }\Gamma\},$$
  $$J=\{(\alpha,d)|x_{\alpha}\in \Omega\text{ and }[x_{\alpha-d}\not\in \Omega \text{   or }x_{\alpha+d}\not\in \Omega]\}.$$

\begin{figure}[!ht]
    \centering
\begin{tikzpicture}[scale=1.2]
\definecolor{black25}{RGB}{25,25,25}
\definecolor{darkpastelpurple}{rgb}{0.59, 0.44, 0.84}

\begin{axis}[
axis line style={gray, line width=0pt},legend cell align={center}, legend columns=3,legend style={at={(0.25,-0.12)},
 /tikz/column 2/.style={column sep = 10pt}, /tikz/column 4/.style={column sep = 10pt}, anchor=south west}, grid=both,minor tick num=3,xmin=0, xmax=0.9, ymin=0.1, ymax=1, tick align=outside, xmajorticks=false, ymajorticks=false, grid style={gray, very thin},  xtick style={white}, ytick style={white}]
\path [draw=darkpastelpurple, fill=darkpastelpurple, opacity=0.3]
    (0.2,0.15)--(0.25,0.15)--(0.3,0.15)--(0.35,0.15)--(0.4,0.15)--(0.45,0.15)--(0.5,0.15)--(0.55,0.15)
    --(0.55,0.2)--(0.6,0.2)--(0.65,0.2)--(0.65,0.25)--(0.65,0.3)--(0.7,0.3)--(0.7,0.35)--(0.7,0.4)
    --(0.7,0.45)--(0.75,0.45)--(0.75,0.5)--(0.75,0.55)--(0.75,0.6)--(0.8,0.6)--(0.8,0.65)--(0.8,0.7)
    --(0.8,0.75)--(0.8,0.8)--(0.75,0.8)--(0.75,0.85)--(0.7,0.85)--(0.7,0.9)--(0.65,0.9)--(0.6,0.9)
    --(0.6,0.95)--(0.55,0.95)--(0.5,0.95)--(0.45,0.95)--(0.4,0.95)--(0.35,0.95)--(0.3,0.95)--(0.25,0.95)
    --(0.25,0.9)--(0.2,0.9)--(0.2,0.85)--(0.15,0.85)--(0.15,0.8)--(0.15,0.75)--(0.1,0.75)--(0.1,0.7)
    --(0.1,0.65)--(0.1,0.6)--(0.1,0.55)--(0.1,0.5)--(0.1,0.45)--(0.1,0.4)--(0.1,0.35)--(0.1,0.3)
    --(0.15,0.3)--(0.15,0.25)--(0.15,0.2)--(0.2,0.2)--cycle;
\addlegendimage{area legend, draw=darkpastelpurple, fill=darkpastelpurple, opacity=0.3}
\addlegendentry{$\overline{\Omega}_h$}

 \addplot [line width=0.3pt, black25, mark=*, mark size=1, mark options={solid}, only marks]
coordinates {
(0.25, 0.2)(0.3, 0.2)(0.35, 0.2)(0.4, 0.2)(0.45, 0.2)(0.5, 0.2)(0.2, 0.25)(0.25, 0.25)
(0.3, 0.25)(0.35, 0.25)(0.4, 0.25)(0.45, 0.25)(0.5, 0.25)(0.55, 0.25)(0.6, 0.25)(0.2, 0.3)
(0.25, 0.3)(0.3, 0.3)(0.35, 0.3)(0.4, 0.3)(0.45, 0.3)(0.5, 0.3)(0.55, 0.3)(0.6, 0.3)
(0.15, 0.35)(0.2, 0.35)(0.25, 0.35)(0.3, 0.35)(0.35, 0.35)(0.4, 0.35)(0.45, 0.35)(0.5, 0.35)
(0.55, 0.35)(0.6, 0.35)(0.65, 0.35)(0.15, 0.4)(0.2, 0.4)(0.25, 0.4)(0.3, 0.4)(0.35, 0.4)(0.4, 0.4)(0.45, 0.4)
(0.5, 0.4)(0.55, 0.4)(0.6, 0.4)(0.65, 0.4)(0.15, 0.45)(0.2, 0.45)(0.25, 0.45)(0.3, 0.45)(0.35, 0.45)(0.4, 0.45)
(0.45, 0.45)(0.5, 0.45)(0.55, 0.45)(0.6, 0.45)(0.65, 0.45)(0.15, 0.5)(0.2, 0.5)(0.25, 0.5)(0.3, 0.5)(0.35, 0.5)
(0.4, 0.5)(0.45, 0.5)(0.5, 0.5)(0.55, 0.5)(0.6, 0.5)(0.65, 0.5)(0.7, 0.5)(0.15, 0.55)(0.2, 0.55)(0.25, 0.55)
(0.3, 0.55)(0.35, 0.55)(0.4, 0.55)(0.45, 0.55)(0.5, 0.55)(0.55, 0.55)(0.6, 0.55)(0.65, 0.55)(0.7, 0.55)
(0.15, 0.6)(0.2, 0.6)(0.25, 0.6)(0.3, 0.6)(0.35, 0.6)(0.4, 0.6)(0.45, 0.6)(0.5, 0.6)(0.55, 0.6)(0.6, 0.6)
(0.65, 0.6)(0.7, 0.6)(0.15, 0.65)(0.2, 0.65)(0.25, 0.65)(0.3, 0.65)(0.35, 0.65)(0.4, 0.65)(0.45, 0.65)
(0.5, 0.65)(0.55, 0.65)(0.6, 0.65)(0.65, 0.65)(0.7, 0.65)(0.75, 0.65)(0.15, 0.7)(0.2, 0.7)(0.25, 0.7)(0.3, 0.7)
(0.35, 0.7)(0.4, 0.7)(0.45, 0.7)(0.5, 0.7)(0.55, 0.7)(0.6, 0.7)(0.65, 0.7)(0.7, 0.7)(0.75, 0.7)(0.2, 0.75)
(0.25, 0.75)(0.3, 0.75)(0.35, 0.75)(0.4, 0.75)(0.45, 0.75)(0.5, 0.75)(0.55, 0.75)(0.6, 0.75)(0.65, 0.75)
(0.7, 0.75)(0.75, 0.75)(0.2, 0.8)(0.25, 0.8)(0.3, 0.8)(0.35, 0.8)(0.4, 0.8)(0.45, 0.8)(0.5, 0.8)(0.55, 0.8)
(0.6, 0.8)(0.65, 0.8)(0.7, 0.8)(0.25, 0.85)(0.3, 0.85)(0.35, 0.85)(0.4, 0.85)(0.45, 0.85)(0.5, 0.85)(0.55, 0.85)
(0.6, 0.85)(0.65, 0.85)(0.3, 0.9)(0.35, 0.9)(0.4, 0.9)(0.45, 0.9)(0.5, 0.9)(0.55, 0.9)(0.25, 0.15)(0.3, 0.15)
(0.35, 0.15)(0.4, 0.15)(0.45, 0.15)(0.5, 0.15)(0.2, 0.2)(0.55, 0.2)(0.6, 0.2)(0.15, 0.25)(0.65, 0.25)(0.15, 0.3)
(0.65, 0.3)(0.1, 0.35)(0.7, 0.35)(0.1, 0.4)(0.7, 0.4)(0.1, 0.45)(0.7, 0.45)(0.1, 0.5)(0.75, 0.5)(0.1, 0.55)
(0.75, 0.55)(0.1, 0.6)(0.75, 0.6)(0.1, 0.65)(0.8, 0.65)(0.1, 0.7)(0.8, 0.7)(0.15, 0.75)(0.8, 0.75)(0.15, 0.8)
(0.75, 0.8)(0.2, 0.85)(0.7, 0.85)(0.25, 0.9)(0.6, 0.9)(0.65, 0.9)(0.3, 0.95)(0.35, 0.95)(0.4, 0.95)(0.45, 0.95)
(0.5, 0.95)(0.55, 0.95)};
\addlegendentry{$\Omega_h$}
\addplot [semithick, red]
coordinates {%
(0.666660243293647, 0.360786724979879)(0.669828819677344, 0.36942849361914)(0.672869125656037, 0.378093929652293)(0.675799784116983, 0.386779594212583)(0.678639033349618, 0.395483698197901)
(0.681404616215632, 0.404205846604395)(0.6841137137393, 0.412946807039277)(0.686782910636354, 0.421708302971459)(0.689428173452767, 0.430492816799353)(0.69206483730785, 0.439303433138877)
(0.694707594260827, 0.44814371755318)(0.697370464506109, 0.457017597562193)(0.700066752564237, 0.465929286057663)(0.702808985487275, 0.474883242963151)(0.705608816716761, 0.483884141696898)
(0.708476900628003, 0.492936882306837)(0.711422736670987, 0.50204664902065)(0.714454471841231, 0.511218988389099)(0.717578663749998, 0.520459935744744)(0.72080000439815, 0.529776187707126)
(0.724120998957477, 0.539175321293465)(0.727541594893656, 0.548666054511789)(0.731058754434772, 0.558258538847896)(0.734665958074891, 0.567964688006705)(0.738352627750032, 0.577798523386113)
(0.742103436782469, 0.587776507777938)(0.745897485064418, 0.597917849559095)(0.749707287738448, 0.608244717065317)(0.75349753790386, 0.618782313032032)(0.757223577943931, 0.629558706413435)
(0.760829525315565, 0.640604286295943)(0.764246003915892, 0.651950641275357)(0.767387468303207, 0.66362856117639)(0.77014920808679, 0.675664725842467)(0.772404342622787, 0.688076451095969)
(0.77400160317799, 0.700863660351231)(0.774765664619919, 0.713997199286341)(0.774503436416817, 0.727403187349951)(0.773021739240567, 0.740945431627691)(0.770161999974485, 0.754413155733877)
(0.765850926413968, 0.76752812772407)(0.76014891925811, 0.77998514133336)(0.753262014160821, 0.791521000788462)(0.745496283264659, 0.801977971022275)(0.737177449113383, 0.81132378143819)(0.728583817672284, 0.819624425545967)
(0.719918687640902, 0.826997884774631)(0.71131450546227, 0.833575711506289)(0.702849972477832, 0.83948113396154)(0.694567529189635, 0.844820343142766)(0.686486765428919, 0.849681135311786)(0.678613494732085, 0.854134721526572)
(0.670945529642557, 0.858238483650052)(0.663476234896077, 0.862038694241607)(0.656196679403884, 0.865572851281171)(0.649096920871591, 0.868871566283737)(0.642166759667355, 0.871960048091298)(0.635396169485822, 0.874859253446954)
(0.628775531637535, 0.877586774372888)(0.622295756508118, 0.880157520980216)(0.615948335741494, 0.882584249267805)(0.609725335602173, 0.884877975034526)(0.603619386366869, 0.887048292552811)(0.597623652834355, 0.88910362491577)
(0.591731785927699, 0.891051426144484)(0.585937897751553, 0.892898334586537)(0.580236505532575, 0.894650303828458)(0.574622509361362, 0.896312699482771)(0.56909114290696, 0.897890386007861)(0.563637950409149, 0.899387789510603)
(0.558258753888175, 0.900808955316423)(0.552949621617299, 0.902157592869321)(0.547706852941574, 0.903437111359472)(0.542526947777532, 0.904650654214224)(0.537406589954318, 0.905801123550332)(0.532342633235926, 0.906891201959601)
(0.527332081375071, 0.907923373392437)(0.522372072948888, 0.90889993891013)(0.51745987286179, 0.909823029667688)(0.512592860388528, 0.910694619521545)(0.507768516150332, 0.911516536573622)(0.50298441322985, 0.912290471783573)
(0.49823821086878, 0.913017986601535)(0.493527646528347, 0.913700520198416)(0.488850528721273, 0.914339395881209)(0.48420472942637, 0.914935826856615)(0.479588178140941, 0.915490920744694)(0.474998857715127, 0.916005683511786)
(0.470434799183607, 0.916481023103952)(0.465894076969149, 0.916917752564743)(0.461374804450134, 0.917316592643138)(0.456875129731144, 0.917678173855036)(0.452393231671137, 0.918003038128817)(0.447927315819921, 0.918291640025468)
(0.443475611110553, 0.918544347428461)(0.439036366761724, 0.918761441834788)(0.434607849041862, 0.918943118253907)(0.430188338194494, 0.919089484636593)(0.425776125456997, 0.919200560788347)(0.421369510223303, 0.919276276823699)
(0.416966797340692, 0.919316471110364)(0.412566294535714, 0.919320887671761)(0.408166309975953, 0.919289173013596)(0.403765149979327, 0.919220872334669)(0.399361116886032, 0.919115425099033)(0.394952507114785, 0.918972159885749)
(0.390537609444538, 0.918790288413015)(0.386114703562656, 0.918568898747866)(0.381682058928189, 0.918306947598515)(0.377237934017069, 0.918003251600881)(0.37278057603145, 0.917656477505031)(0.368308221172597, 0.917265131172717)
(0.36381909597889, 0.916827545395093)(0.359311419328125, 0.916341866172793)(0.354783405020475, 0.91580603735543)(0.350233265937412, 0.915217783665196)(0.345659219737937, 0.914574591899759)(0.34105949643348, 0.913873690183997)
(0.336432348220509, 0.913112025253429)(0.331776062967196, 0.912286238116485)(0.327088981865377, 0.911392637165998)(0.322369517854292, 0.910427168562558)(0.317616180569936, 0.909385384935798)(0.312827607521156, 0.908262413103531)
(0.308002606440958, 0.907052921608421)(0.303140198705584, 0.905751084578466)(0.298239673919762, 0.904350547612811)(0.293300663070422, 0.902844399340219)(0.288323217620451, 0.901225142693705)(0.283307899571732, 0.899484675290265)
(0.27825590242405, 0.897614285547005)(0.273169168071813, 0.895604657503811)(0.268050538641736, 0.893445909399103)(0.262903905830112, 0.891127652001534)(0.257734387837358, 0.888639098512079)(0.252548489169581, 0.885969213193545)(0.247354273358783, 0.883106927288482)
(0.242161498706844, 0.880041418223069)(0.236981711439202, 0.876762456184217)(0.231828279108925, 0.873260823028193)(0.226716323944218, 0.86952878334074)(0.221662537949702, 0.865560585085176)(0.216684862149643, 0.861352953187951)(0.211802022513855, 0.856905525158695)
(0.20703293792606, 0.852221173795012)(0.2023960361165, 0.847306164784334)(0.197908531761957, 0.842170109707589)(0.193585736833421, 0.836825701535747)(0.18944047725135, 0.831288248978238)(0.185482666360241, 0.825575046624288)(0.181719081144783, 0.819704664636117)
(0.178153341659994, 0.813696213592168)(0.174786071633717, 0.807568665119011)(0.171615206788143, 0.801340274197055)(0.168636392929707, 0.795028136812575)(0.165843429429444, 0.788647885122327)(0.163228713912239, 0.782213515862071)(0.160783653932181, 0.775737327188253)
(0.15849903539673, 0.769229954237125)(0.156365327600271, 0.762700450615234)(0.154372926861966, 0.75615641622202)(0.152512352245961, 0.749604152854475)(0.150774380175749, 0.743048809070264)(0.149150150179425, 0.73649454459703)(0.147631225718841, 0.729944665674682)
(0.146209638218091, 0.723401760221545)(0.144877906152952, 0.716867820184552)(0.143629039326221, 0.710344334539453)(0.142456538101941, 0.703832389928742)(0.141354379683345, 0.697332742918027)(0.140317001938581, 0.690845880808211)(0.139339286996771, 0.684372080409257)
(0.138416541413192, 0.67791145621147)(0.137544474719782, 0.671463991418457)(0.136719180359009, 0.665029568639479)(0.135937117267912, 0.658607996099072)(0.135195091354933, 0.652199029236081)(0.134490237942836, 0.645802381459248)(0.133820005088968, 0.639417735289128)
(0.133182138059558, 0.633044750095574)(0.132574665325299, 0.626683067876777)(0.131995884138377, 0.620332318229361)(0.131444347263686, 0.613992117914046)(0.130918850179691, 0.607662070154113)(0.130418418256887, 0.601341763045752)(0.129942295681965, 0.595030767486253)
(0.129489933398657, 0.588728634462777)(0.129060978247366, 0.582434891472551)(0.128655262738533, 0.576149039685547)(0.128272794152208, 0.569870550296849)(0.127913745482216, 0.563598859915711)(0.12757844578711, 0.557333367094935)(0.127267371763364, 0.551073428719002)
(0.126981140191676, 0.544818357049516)(0.126720500303212, 0.538567416720931)(0.12648632811664, 0.532319822228566)(0.126279620666413, 0.526074735778794)(0.126101491830823, 0.519831265337462)(0.125953169149781, 0.513588463228874)(0.125835991411295, 0.507345324915391)
(0.125751408045363, 0.501100788240991)(0.125700979185171, 0.494853732986695)(0.125686377503198, 0.488602980778785)(0.125709391239822, 0.482347295393624)(0.125771928756304, 0.476085383432174)(0.125876024650264, 0.469815895443095)(0.126023847493804, 0.463537427522535)
(0.126217709288014, 0.457248523474294)(0.12646007665413, 0.450947677634166)(0.126753584184621, 0.444633338474431)(0.127101049385208, 0.438303913153594)(0.12750549029526, 0.431957773272)(0.1279701450636, 0.42559326194581)(0.128498493885393, 0.419208702771906)(0.129094284081022, 0.412802410797152)
(0.129761557105484, 0.406372706097697)(0.130504679195192, 0.399917930464981)(0.131328375232045, 0.393436465964279)(0.132237764385687, 0.38692676166512)(0.133238399765096, 0.380387365303317)(0.134336310710235, 0.373816961253809)(0.135538048030774, 0.367214414776156)(0.13685073114053, 0.360578826835991)
(0.138282094816439, 0.353909605880459)(0.139840537175104, 0.347206548698531)(0.141535164944463, 0.340469934361365)(0.143375830339605, 0.333700648921074)(0.145373162464027, 0.326900316355933)(0.147538576514363, 0.320071467176906)(0.149884266200098, 0.313217720493625)
(0.152423156569312, 0.306344000691291)(0.155168816976158, 0.299456766021493)(0.158135309506731, 0.292564267867092)(0.161336966511829, 0.28567680362317)(0.16478807665293, 0.278806963941435)(0.168502466630005, 0.271969844647975)(0.172492972593805, 0.265183188507854)(0.176770807131398, 0.258467413857605)
(0.181344842154637, 0.251845486640896)(0.186220851351384, 0.245342592908093)(0.191400780322159, 0.238985581988368)(0.196882133161412, 0.232802174221085)(0.202657571781603, 0.226819967444234)(0.208714818121767, 0.221065312994573)(0.215036909473386, 0.215562176973856)(0.221602817538974, 0.210331109994374)
(0.22838837057628, 0.205388453819793)(0.23536737735136, 0.200745868430821)(0.242512818485396, 0.196410216976324)(0.249797972666583, 0.192383787040048)(0.257197373753511, 0.188664780794054)(0.264687536125196, 0.185247982000575)(0.272247430392623, 0.182125504291858)(0.27985872591649, 0.179287539280189)
(0.287505838799342, 0.176723045176547)(0.295175833665359, 0.174420341156254)(0.302858236396972, 0.17236759135154)(0.310544781691378, 0.170553181204641)(0.318229142325138, 0.16896599777135)(0.325906658711549, 0.167595622699429)(0.333574082707308, 0.166432462238067)(0.341229335252645, 0.165467824182893)
(0.348871304096029, 0.164693954905619)(0.356499658343791, 0.164104051577122)(0.364114687118326, 0.16369225548363)(0.371717164544196, 0.16345363305794)(0.379308234844941, 0.163384150898463)(0.386889312194408, 0.163480648219998)(0.394461994393358, 0.163740809126907)(0.402027987411536, 0.164163136530596)
(0.409589038070462, 0.164746928847944)(0.417146872374528, 0.165492260174579)(0.424703139560372, 0.166399964686374)(0.432259356533205, 0.167471624507226)(0.439816849446483, 0.168709560431283)(0.447376693791796, 0.170116825639062)(0.454939657745098, 0.171697203839508)(0.462506133490934, 0.173455204989766)
(0.470076056631836, 0.175396061416036)(0.477648834284232, 0.177525724181898)(0.485223239515731, 0.179850848604379)(0.492797319016562, 0.182378777626634)(0.500368265970332, 0.185117506601624)(0.507932298738952, 0.188075633759191)(0.515484520637555, 0.191262285721465)(0.523018768694591, 0.194687009209423)
(0.530527464459191, 0.19835962434653)(0.538001467294414, 0.20229002866584)(0.545429942570514, 0.206487942910262)(0.552800261890923, 0.210962592268847)(0.560097955632531, 0.215722319775668)(0.567306742566844, 0.220774135127229)(0.574408663933852, 0.226123211613264)(0.581384346909779, 0.231772354676064)(0.588213418278866, 0.23772148139117)(0.594875075645076, 0.24396716236065)
(0.601348805749675, 0.250502289589137)(0.607615209844872, 0.257315926341577)(0.613656875775315, 0.264393390196882)(0.619459219500213, 0.271716601514399)(0.625011201686433, 0.279264672363277)(0.630305847309235, 0.28701470729538)(0.635340510127537, 0.294942719345181)(0.640116873860371, 0.303024588646089)
(0.64464069590407, 0.311236945539549)(0.648921359921151, 0.319557923343787)(0.652971299300021, 0.327967730048386)(0.656805361071395, 0.336449018239154)(0.660440175012837, 0.344987067521878)(0.666660243293647, 0.360786724979879)};
\addlegendentry{$\Gamma$}
\end{axis}

\end{tikzpicture}

     \caption{Representation of $\overline{\Omega}_h$, $\Omega_h$ and $\Gamma$.}
    \label{fig:situation}
\end{figure}
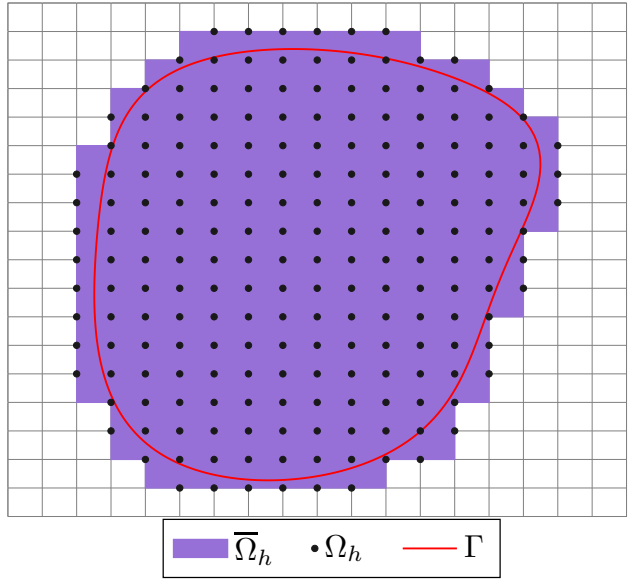

The discrete $L^2$-norm, $L^{\infty}$-norm and $H^1$-semi-norm are defined for all $v_h=(v_{\alpha})_{\alpha:x_{\alpha}\in\Omega_h^{\text{int}}}$ as
\begin{equation*}
\| v_h \|_{h,0} = \left( {h^n} \sum_{\scriptstyle \alpha:x_{\alpha} \in \Omega_h^{\text{int}}} v_{\alpha}^2  \right)^{1 / 2} ,~~~
\| v_h \|_{h,\infty} =\max_{\scriptstyle \alpha:x_{\alpha} \in \Omega_h^{\text{int}}} |v_{\alpha}|
\end{equation*}
and 
\begin{equation*}
|v_h|_{h,1}
= \left(\sum_{
\begin{array}{c}
\scriptstyle \alpha,d:x_{\alpha}\in\Omega^{\text{int}}\\ 
\scriptstyle\text{ and }x_{\alpha+d}\in \Omega^{\text{int}}
\end{array}
}{h^n}\left|\dfrac{v_{\alpha+d}-v_{\alpha}}{h}\right|^2
\right)^{1/2}.
\end{equation*}

Here and after, in the inequalities, $C$ will denote a constant independent on $h$ and $f$.

Let us first define the notion of regularity we will need on the domain:

\begin{definition}
We say that a domain $\Omega$ is $r$-smooth if for each point $x_0\in \Gamma$ there exists a cone centered at $x_0$ with an angle strictly greater than $\pi/2$ and a radius $r$ which is included in $\overline{\Omega}$.
\end{definition}

{Let us now state our convergence result:}

\begin{theorem}[Convergence]\label{theo} 
Suppose that $\Omega$ is $r$-smooth for a $r>0$ and is defined by a level-set function $\varphi\in\mathcal C^2(\overline{\Omega}_h)$ as in \eqref{eq:phi}.
Let $u$ be the solution of the continuous system \eqref{eq:poisson}.
Suppose that $u\in\mathcal C^4(\Omega)$.
For $\sigma,~\gamma$ large enough and $h<\dfrac{2r}{\sqrt{10}}$, the discrete system \eqref{eq:diff} admits a unique solution $u_h$. 
In this case, denoting by $U=(u(x_{\alpha}))_{\alpha:x_{\alpha}\in\Omega_h^{\text{int}}}$, one has
\begin{equation}
    \label{eq:est}
  \|U-u_h\|_{h,0}
+\|U-u_h\|_{h,\infty}
+|U-u_h|_{h,1}
+\leqslant Ch^{3/2}\|u\|_{\mathcal{C}^4(\Omega)}.  
\end{equation}
\end{theorem}
{
\begin{remark}\label{rem:convergence}\begin{itemize}
    \item The $L^2$ convergence rate given in Theorem  \ref{theo} may not be optimal since we numerically observe a second-order rate and standard finite difference schemes are also of second order.  
    \item Concerning the $H^1$ error, the convergence order is larger than the FEM one. This phenomenon is well-known and called supraconvergence (see e.g. 
    \cite{ferreira1998supraconvergence}).
\end{itemize}

\end{remark}}

Moreover, the matrix associated to the discrete system is well-conditioned:
\begin{theorem}[Conditioning]\label{theo:cond}
Under the hypothesis of Theorem \ref{theo},
the condition number defined by $\kappa(A) :=
\|A\|_2\|A^{-1}\|_2$ of the matrix $A$ associated to the bilinear form $a_h$ satisfies
$$
\kappa(A) \leq Ch^{-2} .
$$
Here, $\|\cdot\|_2$ stands for the matrix norm associated with the Euclidian norm.
\end{theorem}

These two theorems are proved in Section \ref{sec:prove}.

{
\begin{remark}
\begin{itemize}
    \item 
    {The following proofs are given in 2D for readability in Section \ref{sec:prove} but can be done in the same way in 3D by adding corresponding indices.} 
    \item In the case of non-homogeneous Dirichlet boundary conditions $u_h^D=(u_{\alpha}^D)_{\alpha}$, one needs to add the following term in the right-hand side:
    \begin{equation*}
{b_h^{rsh}(v_h)} =\frac{\gamma}{2 h^2}\left(\sum_{(\alpha,d) {\in B}} 
\frac{1}{\varphi_{\alpha}^2 + \varphi_{\alpha+d}^2}
(\varphi_{\alpha+d} u_{\alpha}^D - \varphi_{\alpha} u_{\alpha+d}^D) (\varphi_{\alpha+d}
  v_{\alpha} - \varphi_{\alpha} v_{\alpha+d})\right).
  \end{equation*} 
\end{itemize}
\end{remark}}

\section{Link with \texorpdfstring{$\varphi -$}{test}FEM}\label{sec:link}
{Consider $\mathcal{T}_h^{\mathcal{O}}$ a Cartesian triangular {(or tetrahedral in 3D)} mesh of $\mathcal{O}$ with nodes $(x_{\alpha})$, $\mathcal{T}_h$ the set of cells belonging to $\mathcal{T}_h^{\mathcal{O}}$ and intersecting $\Omega$, $\Omega_h^{EF}$ the domain covered by the mesh $\mathcal{T}_h$ and $\partial\Omega_h^{EF}$ its boundary. 
Let $\mathcal{E}_h^{\Gamma}$ the set of edges belonging to $\mathcal{T}_h$ cut by $\Gamma$ and
 { $\mathcal{F}_h^\Gamma$} the set of internal edges belonging to a cell of $\mathcal{T}_h$ cut by $\Gamma$. 
We define
$$V_h=\{v_h\in C^0(\Omega_h):v_{h|K}\in\mathbb{P}_1(K)~\forall K\in \mathcal{T}_h\}$$
and
{$$Q_h=\{p_h\in L^2(E_h^\Gamma):p_{h|K}\in\mathbb{P}_0(E)~\forall E
  \in \mathcal{E}_h^{\Gamma}\},$$
where $E_h^\Gamma=\cup_{E \in \mathcal{E}_h^{\Gamma}} E$.}

Consider the following $\varphi$-FEM scheme for \eqref{eq:poisson}: Find $(u_h,p_h)\in V_h\times Q_h$

\begin{equation}
  \label{phiFEMmod} \int_{\Omega_h} \nabla u_h \cdot \nabla v_h -
  \int_{\partial \Omega_h} \nabla u_h \cdot nv_h + \frac{\gamma}{h } \sum_{E
  \in \mathcal{E}_h^{\Gamma}} \int_E (u_h - \varphi_h p_h) (v_h - \varphi_h q_h)
\end{equation}
\[ + \sigma h \sum_{F \in \mathcal{F}_h^{\Gamma}} \int_F [n \cdot \nabla
   u_h] [n \cdot \nabla v_h] = \int_{\Omega_h} fv_h, \quad \forall v_h, q_h \in V_h\times Q_h.\]
   This $\varphi$-FEM scheme is a variant of the one given in  \cite{cotin2023phi}.
We impose here $u_h \sim \varphi_h p_h$ by penalization on the edges $E \in
\mathcal{E}_h^{\Gamma}$. 
The solution $u_h$ is represented by its values $u_{\alpha}$ at the nodes $x_{\alpha}$. 
If $x_{\alpha}$ is inside $\Omega$ together with all its  neighbors,
then \eqref{phiFEMmod} gives (after division by $h^2$ and some quadrature) the
usual discretization
{\begin{equation}
  \label{FDin} \sum_{d\in D}\frac{- u_{\alpha-d}+2 u_{\alpha}  - u_{\alpha+d}}{h^2} = f_{\alpha}.
\end{equation}}
Like this, we have the equations at the interior nodes, but the active
unknowns are also at the nodes outside $\Omega$ but adjacent to an inside
node.
If $v_h$ is a basis function attached to such a node, then we
simply ignore the contribution $\int_{\Omega_h} \nabla u_h \cdot \nabla v_h -
\int_{\partial \Omega_h} \nabla u_h \cdot nv_h$ (and also $\int_{\Omega_h}
fv_h$ on the right-hand side). On the contrary, we want to keep the equations coming from 
\begin{equation}
  \label{gammaterm} \frac{\gamma}{h^3} \sum_{E \in \mathcal{E}_h^{\Gamma}}
  \int_E (u_h - \varphi_h p_h) (v_h - \varphi_h q_h),
\end{equation}
which we have divided by $h^2$ to be consistent with (\ref{FDin}). For any $E
\in \mathcal{E}_h^{\Gamma}$, $p_h$ and $q_h$ on $E$ are just some numbers, say
$p_E$ and $q_E$. Taking $v_h = 0$ in (\ref{gammaterm}) gives
\[ \int_E (u_h - \varphi_h p_E) \varphi_h = 0, \]
thus
\[ p_E = \frac{\int_E u_h \varphi_h}{\int_E \varphi_h^2}. \]
We can now take $q_h = 0$ and exclude $p_h$ from (\ref{gammaterm}), which
becomes
\begin{equation}
  \label{gammatSimp} \frac{\gamma}{{h^3} } \sum_{E \in \mathcal{E}_h^{\Gamma}}
  \int_E \left( u_h - \frac{\int_E u_h \varphi_h}{\int_E \varphi_h^2} \varphi_h \right)
  v_h.
\end{equation}
Let us work out this term in the case when $E \in \mathcal{E}_h^{\Gamma}$ is
an edge from $x_{\alpha}$ to $x_{\alpha+d}$,  {with $x_\alpha$ inside $\Omega$ and $x_{\alpha +d}$ outside and $d\in D$.} In this case
\[ u_h - \frac{\int_E u_h \varphi_h}{\int_E \varphi_h^2} \varphi_h =
   \left\{\begin{array}{l}
     \frac{\varphi_{{\alpha}+d}}{\varphi_{{\alpha}}^2 + \varphi_{{\alpha}+d}^2} (\varphi_{{\alpha}+d}
     u_{{\alpha}} - \varphi_{{\alpha}} u_{{\alpha}+d}) \text{ at } x_{{\alpha}},\\
     \frac{\varphi_{{\alpha}}}{\varphi_{{\alpha}}^2 + \varphi_{{\alpha+d}}^2} (\varphi_{{\alpha}} u_{\alpha+d}
     - \varphi_{\alpha+d} u_{{\alpha}}) \text{ at } x_{\alpha+d}
   \end{array}\right. ,\]
so the contribution to (\ref{gammatSimp}) on this edge $E$ is given by
\begin{equation*}
 \frac{\gamma}{2 h^2} \frac{1}{\varphi_{\alpha}^2 + \varphi_{\alpha +d}^2} (\varphi_{{\alpha}+d}
     u_{{\alpha}} - \varphi_{{\alpha}} u_{{\alpha}+d}) (\varphi_{{\alpha}+d}
     v_{{\alpha}} - \varphi_{{\alpha}} v_{{\alpha}+d}),
\end{equation*}
which is of the same order as the penalization term $b_h$.

Similar formulas hold for other configurations of edges $E \in
\mathcal{E}_h^{\Gamma}$. This gives the matrix representing (\ref{gammaterm}),
which should be added to the matrix representing (\ref{FDin}).

Finally, the ghost penalty term
\begin{equation}
  \label{ghost} \sigma h\sum_{F \in \mathcal{F}_h^{\Gamma}} \int_F [n \cdot
  \nabla u_h] [n \cdot \nabla v_h],
\end{equation}
{which will also be divided by }
$h^2$ can also be approximated in a Finite Difference manner.
Take again a node $x_{\alpha}$ inside $\Omega$ such that $x_{\alpha+d}$ is
outside $\Omega$ {with $d\in D$.}
{Then the two edges $(x_{\alpha-d}-x_{\alpha}$ and $(x_{\alpha}-x_{\alpha+d})$ adjacent to $x_{\alpha}$ are in $\mathcal{F}_h^{\Gamma}$}
and the above contributions on these edges can be
approximated as
\begin{equation*}
  \sigma \frac{- u_{\alpha -d} + 2 u_{\alpha} - u_{\alpha+d}}{h }
  \times \frac{- v_{\alpha-d} + 2 v_{\alpha} - v_{\alpha+d}}{h }.
\end{equation*}

\section{Proof of Theorems \ref{theo} and \ref{theo:cond}}\label{sec:prove}

Most studies in the literature \cite{LiEF,jovanovic2013analysis} analyze the finite difference methods using the formalism of finite elements or finite volumes \cite{Johansen}  on elliptic problems. {We will follow here the finite element formalism.}

Let us introduce the following discrete $L^2$-norm, $L^{\infty}$-norm and $H^1$-semi-norm on $\Omega_h$ defined for all $v_h=(v_{\alpha})_{\alpha:x_{\alpha}\in \Omega_h}$ as
\begin{equation*}
\| v_h \|_{h,0,\Omega_h} = \left( h^2 \sum_{\scriptstyle \alpha:x_{\alpha} \in \Omega_h} v_{\alpha}^2  \right)^{1 / 2} ,~~~
\| v_h \|_{h,\infty,\Omega_h} =\max_{\scriptstyle \alpha:x_{\alpha} \in \Omega_h} |v_{\alpha}|
\end{equation*}
and 
\begin{equation*}
|v_h|_{h,1,\Omega_h}
= \left(\sum_{
\begin{array}{c}
\scriptstyle \alpha,d:x_{\alpha}\in\Omega\\ 
\scriptstyle\text{ or }x_{\alpha+d}\in \Omega
\end{array}
}h^2\left|\dfrac{v_{\alpha+d}-v_{\alpha}}{h}\right|^2
\right)^{1/2}.
\end{equation*}

\color{black}

{
We will focus here on the 2D case, but the reader will see that the other situation can be treated similarly. 
{In this case}, the problem can be rewritten as follows: find a discrete function $u_h=(u_{ij})_{ij}$ defined on $\Omega_h$ such that}
\begin{equation*}
a_h(u_h,v_h)=l_h(v_h),
\end{equation*}
{for all discrete function $v_h=(v_{ij})_{ij}$ defined on $\Omega_h$,} 
where
$$
a_h(u_h,v_h)=(-\Delta_h u_h,v_h)+b_h(u_h,v_h)+j_h(u_h,v_h)\,,
$$
and 
$$
l_h(v_h)=\sum_{i,j} f_{ij}v_{ij}\,,
$$
with
    the discrete Laplacian:
{\begin{equation*}
(-\Delta_h u_h,v_h)=\sum_{\scriptstyle i,j|(x_i,y_j)\in \Omega}\frac{4 u_{ij} - u_{i - 1, j} - u_{i + 1, j} - u_{i, j - 1} -
  u_{i, j + 1}}{h^2}v_{ij},
  \end{equation*}}
a penalization for the boundary conditions
\begin{multline*}
b_h(u_h,v_h) =\frac{\gamma}{ h^2}\left(\sum_{
(i,j)\in B_x} 
\frac{1}{\varphi_{ij}^2 + \varphi_{i + 1,  j}^2}
(\varphi_{i + 1, j} u_{ij} - \varphi_{ij} u_{i + 1, j}) (\varphi_{i + 1, j}
  v_{ij} - \varphi_{ij} v_{i + 1, j})\right. \\
 \left. 
 +\sum_{(i,j)\in B_y} \frac{1}{\varphi_{ij}^2+ \varphi_{i ,  j+1}^2}
 (\varphi_{i , j+1} u_{ij} - \varphi_{ij} u_{i, j+1}) (\varphi_{i , j+1}
  v_{ij} - \varphi_{ij} v_{i, j+1})\right)
  \end{multline*} 
 and a stabilization term near the boundary
\begin{multline*}
j_h(u_h,v_h)= \sigma\bigg( \sum_{(i,j)\in J_x}\frac{- u_{i - 1, j} + 2 u_{ij} - u_{i + 1, j}}{h }
  \times \frac{- v_{i - 1, j} + 2 v_{ij} - v_{i + 1, j}}{h }
  \\
 +\sum_{(i,j)\in J_y} \frac{- u_{i , j- 1} + 2 u_{ij} - u_{i , j+ 1}}{h }
  \times \frac{- v_{i , j- 1} + 2 v_{ij} - v_{i , j+ 1}}{h }
  \bigg)
\end{multline*}
{
  with $\gamma,~\sigma>0$ and
  $$B_x=\{(i,j)|\text{ the edge }(x_i, y_j) - (x_{i + 1}, y_j)\text{   intersects }\Gamma\text{ and not included in }\Gamma\},$$
    $$B_y=\{(i,j)|\text{ the edge }(x_i, y_j) - (x_{i}, y_{j+1})\text{   intersects }\Gamma\text{ and not included in }\Gamma\},$$
  $$J_x=\{(i,j)|(x_i, y_j)\in \Omega\text{ and }[(x_{i - 1}, y_j)\not\in \Omega \text{   or }(x_{i + 1}, y_j)\not\in \Omega]\},$$
  and
  $$J_y=\{(i,j)|(x_i, y_j)\in \Omega\text{ and }[(x_{i }, y_{j-1})\not\in \Omega \text{   or }(x_{i }, y_{j+1})\not\in \Omega]\}.$$
    }

We now give some intermediate results, which will be used to prove Theorems \ref{theo} and \ref{theo:cond}.

The first one is an adaptation of Lemma 3.3 in  \cite{duprez}, which will be central in the proof of the convergence.

\begin{lemma}\label{lemma:magic}
There exist  $\alpha_1\in(0,1)$, $\alpha_2\in(0,1/2)$ and $\beta>0$ such that
$$\left|\frac{u_{1}-u_0}{h}\right|^2
\leqslant \alpha_1
\left|\frac{u_{1}-u_0}{h}\right|^2
+\alpha_2\left|\frac{u_{2}-u_{1}}{h}\right|^2
+\beta \left|\frac{u_{0}-2u_1+u_{2}}{h}\right|^2$$
for all $u_{0},u_1,u_{2}\in\mathbb{R}$.

\end{lemma}
\begin{proof}
For all $a,~b\in \mathbb{R}$ and  $\varepsilon,~\delta>0$, it holds 
\begin{align*}
a^2&\leq|a|(|a-b|+|b|)
\leq \dfrac{1}{2\varepsilon}a^2+\dfrac{\varepsilon}{2}(|a-b|+|b|)^2\smallskip\\
&\leq \dfrac{1}{2\varepsilon}a^2+\dfrac{\varepsilon}{2}b^2+\varepsilon|a-b||b|+\dfrac{\varepsilon}{2}(a-b)^2\smallskip\\
&\leq \dfrac{1}{2\varepsilon}a^2+\dfrac{\varepsilon}{2}(1+\delta)b^2+\dfrac{\varepsilon}{2}(1+\dfrac{1}{\delta})(a-b)^2.
\end{align*}
For $\varepsilon=\dfrac34$ and $\delta = \dfrac16$, we have
$$a^2\leq \dfrac{2}{3}a^2+\dfrac{7}{16}b^2+\dfrac{\varepsilon}{2}(1+\dfrac{1}{\delta})(a-b)^2,
$$
which leads to the conclusion.
\end{proof}

\begin{lemma}\label{lemma:magic2}
For all $\beta>0$, there exists  $\alpha\in(0,1)$ such that
for all $u_{ij}\in\mathbb{R}$ 
\begin{multline}\label{eq:ine magic 1}
\left|\frac{u_{10}-u_{00}}{h}\right|^2+\left|\frac{u_{20}-u_{10}}{h}\right|^2
\leqslant \alpha\left(
\left|\frac{u_{10}-u_{00}}{h}\right|^2+\left|\frac{u_{20}-u_{10}}{h}\right|^2
\right)\\
+\beta\left(
\left|\frac{u_{11}-u_{01}}{h}\right|^2
+\left|\frac{u_{11}-u_{10}}{h}\right|^2
+\left|\frac{u_{01}-u_{02}}{h}\right|^2\right.\\\left.
+ \left|\frac{u_{00}-2u_{10}+u_{20}}{h}\right|^2
+ \left|\frac{u_{00}-2u_{01}+u_{02}}{h}\right|^2\right).
\end{multline}
\end{lemma}

\begin{proof}{
We only need to prove that for all $\beta>0$, there exists  $\alpha\in(0,1)$ such that
for all $u_{ij}\in\mathbb{R}$
\begin{multline}\label{eq:ine magic 2}
\left|\frac{u_{10}-u_{00}}{h}\right|^2+\left|\frac{u_{20}-u_{10}}{h}\right|^2
\leqslant \alpha\left(
\left|\frac{u_{10}-u_{00}}{h}\right|^2+\left|\frac{u_{20}-u_{10}}{h}\right|^2
\right)\\
+\beta\left(
\alpha\left|\frac{u_{11}-u_{01}}{h}\right|^2
+\alpha\left|\frac{u_{11}-u_{10}}{h}\right|^2
+\alpha\left|\frac{u_{01}-u_{02}}{h}\right|^2\right.\\\left.
+ \left|\frac{u_{00}-2u_{10}+u_{20}}{h}\right|^2
+ \left|\frac{u_{00}-2u_{01}+u_{02}}{h}\right|^2\right).
\end{multline}
Indeed, the right-hand side of \eqref{eq:ine magic 2} is smaller than the right-hand side of \eqref{eq:ine magic 1}.
    Consider 
 \begin{equation}
     \label{eq:sup alpha}
 \alpha =\sup \dfrac{
  \left|u_{10}-u_{00}\right|^2
  +\left|u_{20}-u_{10}\right|^2
 -\beta\left(
 \left|u_{00}-2u_{10}+u_{20}\right|^2
+ \left|u_{00}-2u_{01}+u_{02}\right|^2\right)
}{  \left|u_{10}-u_{00}\right|^2
  +\left|u_{20}-u_{10}\right|^2
+\beta\left|u_{11}-u_{01}\right|^2
+\beta\left|u_{11}-u_{10}\right|^2
+\beta\left|u_{01}-u_{02}\right|^2
},
 \end{equation}
where $$ D:=\left|u_{10}-u_{00}\right|^2
  +\left|u_{20}-u_{10}\right|^2
  +\beta\left|u_{11}-u_{01}\right|^2
+\beta\left|u_{11}-u_{10}\right|^2
+\beta\left|u_{01}-u_{02}\right|^2
\neq0.$$
Without loss of generality, we can assume that 
\begin{equation}\label{eq:magicneq2bis}
    \left|u_{10}-u_{00}\right|^2
  +\left|u_{20}-u_{10}\right|^2
    +\beta\left|u_{11}-u_{01}\right|^2
+\beta\left|u_{11}-u_{10}\right|^2
+\beta\left|u_{01}-u_{02}\right|^2=1
\end{equation}
and
\begin{equation}\label{eq:sum uij}
\sum_{i,j}u_{ij}=0.
\end{equation}
Indeed, 
if $D\neq 1$, we can divide all $u_{ij}$ by $D$. Moreover, we can subtract $\sum_{i,j}u_{ij}$ to all $u_{ij}$ in \eqref{eq:sup alpha} without changing the definition of $\alpha$, so the $u_{ij}$ can be taken such that \eqref{eq:sum uij} holds.
We clearly have $ \alpha\leq 1$.}

{Let us highlight that the set of $u_{ij}$ satisfying \eqref{eq:magicneq2bis} and \eqref{eq:sum uij} is uniformly bounded.
We denote by $\bar u_{ij}=u_{ij}-u_{00}$ for all $(i,j)\neq (0,0)$. 
Then \eqref{eq:magicneq2bis} can be written
\begin{equation*}
    \left|\bar u_{10}\right|^2
  +\left|\bar u_{20}-\bar u_{10}\right|^2
    +\beta\left|\bar u_{11}-\bar u_{01}\right|^2
+\beta\left|\bar u_{11}-\bar u_{10}\right|^2
+\beta\left|\bar u_{01}-\bar u_{02}\right|^2=1.
\end{equation*}
We deduce that
\begin{equation}\label{eq:estim bar uij}
\left|\bar u_{10}\right|\leqslant 1,~ 
\left|\bar u_{20}\right|\leqslant 2,~
\left|\bar u_{11}\right|\leqslant 1+1/\beta,~
\left|\bar u_{01}\right|\leqslant 1+2/\beta,~
\left|\bar u_{02}\right|\leqslant 1+3/\beta.
\end{equation}
Using \eqref{eq:sum uij} and triangular inequality,
$$|6u_{00}|
=\left|\sum_{(i,j)}(u_{ij}-u_{00})\right|
=\left|\sum_{(i,j)}\bar u_{ij}\right|
\leqslant 6+6/\beta.$$
Hence, $|u_{00}|\leqslant 1+1/\beta$.
Moreover, using \eqref{eq:estim bar uij}, 
$|\bar u_{ij}|\leqslant 2+3/\beta$. 
Thus
$$|u_{ij}|\leqslant 3+4/\beta$$ 
for all $(i,j)$.
}

{Since the set of $u_{ij}$ satisfying \eqref{eq:magicneq2bis} and \eqref{eq:sum uij} is closed, bounded and finite-dimensional, the supremum in the definition of $\alpha$ is reached. }

Assume that $ \alpha=1$. {Since $\beta\neq0$,  }
there exists $u_{ij}$ such that 
$$
\left|u_{11}-u_{01}\right|^2
+\left|u_{11}-u_{10}\right|^2
+\left|u_{01}-u_{02}\right|^2
+ \left|u_{00}-2u_{10}+u_{20}\right|^2
+ \left|u_{00}-2u_{01}+u_{02}\right|^2
=0.
$$
We deduce that $u_{11}=u_{01}=u_{10}=u_{02}$, then
$$
  \left|u_{00}-2u_{10}+u_{20}\right|^2
+ \left|u_{00}-u_{10}\right|^2
=0.$$
 Hence $u_{00}=u_{10}=u_{20}$ which is in contradiction with \eqref{eq:magicneq2bis}.
  \end{proof}

Lemmas \ref{lemma:magic} and \ref{lemma:magic2} allow us to deduce the coercivity 
of the bilinear form $a_h$:

\begin{proposition}[Coercivity]\label{prop:coer}
There exists $c>0$ such that, for each $u_h$,
$$a_h(u_h,u_h)\geqslant c |||u_h|||_h^2,$$
where 
{$$|||u_h|||_h=\left(
\frac{1}{h^2}|u_h|_{h,1,\Omega_h}^2
+b_h(u_h,u_h)+j_h(u_h,u_h)\right)^{1/2}.
$$}
\end{proposition}
In the following proof and in the rest of the manuscript, we will use the following notation for each $i,j$
\begin{equation}
    \label{eq:bar u}
 u_{(i, j)-(i+1,j)}^{\varphi} = \dfrac{\varphi_{i, j} u_{i+1, j} - \varphi_{i+1, j} u_{i, j}}{\varphi_{i, j}   - \varphi_{i+1, j}}.
\end{equation}

\begin{proof}[Proof of Proposition \ref{prop:coer}]
     Let us fix the index $j$, and assume that the nodes $(x_i,y_j)$ belonging to $\Omega_h$ are for $i\in\{M_j,\dots,N_j\}$. Without loss of generality, we can assume that $M_j=0$.
    \smallskip

{
\begin{figure}[!ht]
    \centering
\begin{tikzpicture}
\draw (0,0) circle[radius=2pt];
  \foreach \Point/\PointLabel in {(0,0)/u_{00}, (2,0)/u_{10}, (4,0)/u_{20}}
\draw[fill=black] \Point circle (0.08) node[above right] {$\PointLabel$};
\draw [dashed,very thick] (0,0) -- (4,0);
\node at (1.5,1) {$\Gamma$};
\node at (3,1) {$\Omega$};
\draw[very thick] (1.2,1.2) arc (140:220:2cm);
\end{tikzpicture}
    \caption{{Case $N_j>2$ in the proof of Proposition \ref{prop:coer}.}}
    \label{fig:lem1}
\end{figure}
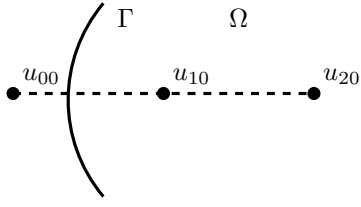
}

\noindent\textbf{Case $\boldsymbol{N_j>2}$:}
{We are in the situation described in Figure \ref{fig:lem1}.}
We remark that 
\begin{multline*}
\sum_{i=1}^{N_j-1}(-u_{i-1,j}+2u_{i,j}-u_{i+1,j})u_{i,j}
\\=-\underbrace{(u_{0,j}-u_{1,j})u_{0,j}}_{(I)}+\underbrace{(u_{N_j-1,j}-u_{N_j,j})u_{N_j,j}}_{(II)}+\sum_{i=0}^{N_j-1}|u_{i+1,j}-u_{ij}|^2.
\end{multline*}
Let us first estimate the term (I). 
Using notation \eqref{eq:bar u}, we remark that
\begin{multline}\label{eq:interpol}
u_{0, j} = 
\dfrac{\sqrt{\varphi_{0, j}^2 + \varphi_{1, j}^2}}{\varphi_{0, j} - \varphi_{1, j}}\dfrac{u_{0, j}\varphi_{0, j} - u_{0, j}\varphi_{1, j}}{\sqrt{\varphi_{0, j}^2 + \varphi_{1, j}^2}}\\
=\dfrac{\sqrt{\varphi_{0, j}^2 + \varphi_{1, j}^2}}{\varphi_{0, j} - \varphi_{1, j}}\left(u_{(0, j)-(1,j)}^{\varphi} + \dfrac{\varphi_{0, j} }{\sqrt{\varphi_{0, j}^2 + \varphi_{1, j}^2}}   (u_{0, j} - u_{1, j})\right) .
   \end{multline}
Since $\varphi_{0, j}\geq 0$ and $\varphi_{1, j} < 0$, 
we have 
\begin{equation}\label{eq:phi<0}
0 \leqslant\dfrac{\varphi_{0, j} }{\sqrt{\varphi_{0, j}^2 + \varphi_{1, j}^2}} < 1
\text{ and }
\dfrac{\sqrt{\varphi_{0, j}^2 + \varphi_{1, j}^2}}{\varphi_{0, j} - \varphi_{1, j}}\leq 1.
\end{equation}
Hence
\begin{equation*}
     (I)\leq |(u_{0, j} - u_{1, j}) u_{(0, j)-(1,j)}^{\varphi}| + (u_{0, j}
   - u_{1, j})^2 .
\end{equation*}   
Moreover, using Young inequality with $\varepsilon > 0$ and Lemma \ref{lemma:magic} with $\alpha_1 \in (0, 1)$, $\alpha_2 \in (0, 1/2)$ and $\beta > 0$, we observe
\begin{multline*}
     (I)\leq \frac{1}{2 \varepsilon} (u_{(0, j)-(1,j)}^{\varphi})^2 +
   \left( 1 + \frac{\varepsilon}{2} \right) (u_{0, j} - u_{1, j})^2 \\ \leqslant
   \frac{1}{2 \varepsilon} (u_{(0, j)-(1,j)}^{\varphi})^2 + \left( 1 + \frac{\varepsilon}{2}
   \right)  (\alpha_1| u_{1, j} - u_{0, j} |^2 + \alpha_2| u_{2, j} - u_{1, j} |^2) \\+
   \left( 1 + \frac{\varepsilon}{2} \right) \beta | u_{2} - 2 u_{1} + u_{0} |^2 .
\end{multline*}
Similarly, it holds
\begin{multline*}
     (II)  \leqslant
   \frac{1}{2 \varepsilon} (u_{(N_j-1, j)-(N_j,j)}^{\varphi})^2 \\
   + \left( 1 + \frac{\varepsilon}{2}
   \right)  (\alpha_1| u_{N_j-1, j} - u_{N_j, j} |^2 + \alpha_2| u_{N_j-2, j} - u_{N_j-1, j} |^2) \\+
   \left( 1 + \frac{\varepsilon}{2} \right) \beta | u_{N_j-2} - 2 u_{N_j-1} + u_{N_j} |^2 .
\end{multline*}
Since $N_j>2$, denoting by $\alpha=\max\{\alpha_1,2\alpha_2\}$, one has
\begin{multline*}
\sum_{i=1}^{N_j-1}\dfrac{(-u_{i-1,j}+2u_{i,j}-u_{i+1,j})u_{i,j}}{h^2}
\geqslant
   \left(1-\alpha\left(1+\frac{\varepsilon}{2}\right)\right) \sum_{i=0}^{N_j-1}\left|\frac{u_{i+1,j}-u_{ij}}{h}\right|^2\\
  -\frac{1}{2\varepsilon}\sum_{i=0}^{N_j-1}\dfrac{(u_{(i, j)-(i+1,j)}^{\varphi})^2}{h^2}
  -\left(1+\frac{\varepsilon}{2}\right)\beta\sum_{i=1}^{N_j-1}\left|\frac{- u_{i - 1, j} + 2 u_{ij} - u_{i + 1, j}}{h }\right|^2.
\end{multline*}

\smallskip

{
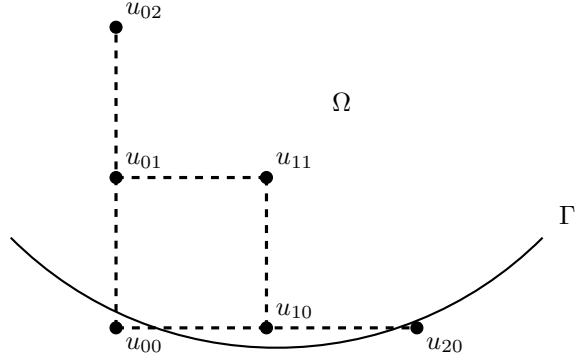
\begin{figure}[!ht]
    \centering
\begin{tikzpicture}
\draw (0,0) circle[radius=2pt];
  \foreach \Point/\PointLabel in { (2,0)/u_{10},(0,2)/u_{01}, (2,2)/u_{11}, (0,4)/u_{02}}
\draw[fill=black] \Point circle (0.08) node[above right] {$\PointLabel$};
  \foreach \Point/\PointLabel in {(0,0)/u_{00}, (4,0)/u_{20}}
\draw[fill=black] \Point circle (0.08) node[below right] {$\PointLabel$};
\draw [dashed,very thick] (0,0) -- (4,0);
\draw [dashed,very thick] (0,0) -- (0,4);
\draw [dashed,very thick] (2,2) -- (2,0);
\draw [dashed,very thick] (2,2) -- (0,2);
\node at (3,3) {$\Omega$};
\draw[thick] (-1.4,1.2) arc (225:315:5cm);
\node at (6,1.5) {$\Gamma$};
\end{tikzpicture}
    \caption{{Case $N_j=2$ in the proof of Proposition \ref{prop:coer}.}}
    \label{fig:lem2}
\end{figure}
}

\noindent\textbf{Case $\boldsymbol{N_j=2}$:}
One has
\begin{multline*}
(-u_{0,j}+2u_{1,j}-u_{2,j})u_{1,j}
\\=-(u_{0,j}-u_{1,j})u_{0,j}+(u_{1,j}-u_{2,j})u_{2,j}+|u_{1,j}-u_{0j}|^2+|u_{2,j}-u_{1j}|^2\\
\leq |(u_{0, j} - u_{1, j}) u_{(0, j)-(1,j)}^{\varphi}| 
+ (u_{0, j}   - u_{1, j})^2 +
   |(u_{2, j} - u_{1, j}) u_{(1, j)-(2,j)}^{\varphi} |
   + (u_{2, j}   - u_{1, j})^2 .
\end{multline*}
We have $(x_0,y_j),(x_2,y_j)\not\in\Omega$ and $(x_1,y_j)\in\Omega$.
The circle containing $(0,0)$, $(2h,0)$, $(0,2h)$ has a radius equal to $\dfrac{\sqrt{10}}{2}h$. 
Then, 
since $\Omega$ is $r$-smooth,
for $h<\dfrac{2r}{\sqrt{10}}$, 
without loss of generality, 
we can assume that we are in the situation described in Figure \ref{fig:lem2}.
Thanks to Lemma \ref{lemma:magic2}, we obtain the same conclusion as in the previous case.

\smallskip

\noindent    \textbf{Conclusion:} Combining the two cases,
\begin{multline*}
    a_h(u_h,u_h)\geqslant
   \left(1-\alpha\left(1+\frac{\varepsilon}{2}\right)\right) \left(\sum_{i,j}\left|\frac{u_{i+1,j}-u_{ij}}{h}\right|^2
    +\sum_{j,i}\left|\frac{u_{i,j+1}-u_{ij}}{h}\right|^2\right)\\
 + \left(1-\frac{1}{2\varepsilon\gamma}
 \right)  b_h(u_h,u_h)
  +\left(1-\left(1+\dfrac{\varepsilon}{2}\right)\frac{\beta}{\sigma}
  \right) j_h(u_h,u_h),
\end{multline*}
which leads to the result taking $\varepsilon$ such that $\alpha \left( 1 + \frac{\varepsilon}{2} \right) < 1$ and then $\gamma$, $\sigma$ large enough.
\end{proof}

\begin{remark}
As seen in the above proof, the assumption on $h$ in Theorem \ref{theo} can be replaced by the two assumptions :
\begin{itemize}
    \item If $(x_{i+1},y_{j}){,}\ (x_{i-1},y_{j})\not\in \Omega\text{  and }(x_{i},y_{j})\in \Omega$ then there exists $k,l\in\{-1,1\}$ such that
$(x_{i+k},y_{j+l})$, $(x_{i+k},y_{j+2l})$, $(x_{i},y_{j+l})\in \Omega.$
\item 
If $(x_{i},y_{j+1}){,}\ (x_{i},y_{j-1})\not\in \Omega\text{  and }(x_{i},y_{j})\in \Omega$ then there exists $k,l\in\{-1,1\}$ such that
$(x_{i+k},y_{j+l})$, $(x_{i+2k},y_{j+l})$, $(x_{i+k},y_{j})\in \Omega.$
\end{itemize}
\end{remark}

We will also need the following Poincaré estimate :
\begin{lemma}\label{lemma2}
 There exists $C_P>0$ such that for each $v_h=(v_{ij})_{ij}$, 
    $$\|v_h\|_{h,\infty,\Omega_h}^2+\|v_h\|_{h,0,\Omega_h}^2
    \leq C_P \left(|v_h|_{h,1,\Omega_h}^2+ h^3b_h(v_h,v_h)\right).$$
\end{lemma}
\begin{proof}
 Let us fix the index $j$, and assume that the first and the last term $(x_i,y_j)$ belonging to $\Omega_h$ are for $i\in\{M_j,\dots,N_j\}$. Without loss of generality, we can assume that $M_j=0$.
We have for all $i$
$$v_{ij}=v_{0j}+\sum_{k=0}^{i-1}(v_{k+1,j}-v_{kj}).$$
Then
$$v_{ij}^2\leq 2v_{0j}^2+2(i-1)\sum_{k=0}^{i-1}(v_{k+1,j}-v_{kj})^2.$$
Denoting by $L$ the maximum of the diameters of the set $\Omega_h$ {(i.e. the biggest distance between two points of $\Omega_h$)},
$N_j\leqslant C L/h$ ($C>0$), we deduce that
$$\sum_{i=0}^{N_j}v_{ij}^2\leq 2C\dfrac{L}{h}v_{0j}^2+2C^2\dfrac{L^2}{h^2}\sum_{i=0}^{N_j-1}(v_{i+1,j}-v_{ij})^2.$$
Using \eqref{eq:interpol} and \eqref{eq:phi<0}, 
$$v_{0, j}^2 \leq 2( u_{(i, j)-(i+1,j)}^{\varphi} )^2 + 2   (v_{0, j} - v_{1, j})^2, $$
which leads to the conclusion.

\end{proof}

\begin{proof}[Proof of Theorem \ref{theo}]

Let us now prove Theorem \ref{theo}.
We remark that there exists $C_0>0$ such that for all $f\in C^2(\Omega)$ and all $h<h_0$ with $h_0>0$, there exists an extrapolation $\tilde u\in \mathcal{C}^4$ of the solution $u$ of \eqref{eq:poisson} such that
\begin{equation}\label{eq:extra}
\|\tilde u\|_{\mathcal{C}^4(\overline{\Omega}_h)}\leqslant C_0 \|u\|_{\mathcal{C}^4(\Omega)}.
\end{equation}
Consider $\tilde u$ such an extrapolation. 
We denote by $\tilde f=-\Delta \tilde u$ 
and $\tilde U=(\tilde u_{ij})_{ij}=(\tilde u(x_{i},y_{j}))_{ij}$.

Let us denote by $e_{ij}=\tilde u_{ij}-u_{ij}$ and $e_h=(e_{ij})_{ij}$.
Thanks to Proposition \ref{prop:coer}, it holds
$$|||e_h|||^2_h\leq \frac{1}{c} a_h(e_h,e_h).$$
Since $u_h$ is solution to \eqref{eq:diff},
$$
a_h(u_h,e_h)
=\sum_{ij}f_{ij}e_{ij}.
$$
Thus
\begin{multline*}
a_h(e_h,e_h)
=\underbrace{-\sum_{\scriptstyle i,j|(x_i,y_j)\in \Omega}\left(-\frac{4\tilde u_{ij} -\tilde  u_{i - 1, j} -\tilde  u_{i + 1, j} -\tilde  u_{i, j - 1} -
\tilde   u_{i, j + 1}}{h^2}-f_{ij}\right)e_{ij}}_{(I)}\\+\underbrace{b_h(\tilde U,e_h)}_{(II)}+\underbrace{j_h(\tilde U,e_h)}_{(III)}.
\end{multline*}
Let us estimate each term:

\textbf{Term (I):} 
Thanks to Cauchy-Schwarz inequality,
$$(I)\leqslant \sqrt{\sum_{\scriptstyle i,j|(x_i,y_j)\in \Omega}\left(-\frac{4 \tilde u_{ij} - \tilde u_{i - 1, j} -\tilde  u_{i + 1, j} -\tilde  u_{i, j - 1} -\tilde  u_{i, j + 1}}{h^2}-f_{ij}\right)^2}
\times \sqrt{\sum_{\scriptstyle i,j|(x_i,y_j)\in \Omega}e_{ij}^2}.$$
There exist $(\xi_{i},\nu_j)\in [x_i-h,x_i+h]\times [y_j-h,y_j+h]$ such that
$$-\frac{4\tilde u_{ij} -\tilde  u_{i - 1, j} -\tilde  u_{i + 1, j} -\tilde  u_{i, j - 1} -
\tilde   u_{i, j + 1}}{h^2}= f_{ij}-\frac{h^2}{12}\left(\frac{\partial^4 \tilde u}{\partial x^4}(\xi_i,y_{j})+\frac{\partial^4 \tilde u}{\partial y^4}(x_i,\nu_{j})\right).$$
Since the number of nodes in $\Omega_h$ is of order $1/h^2$, we deduce that
\begin{multline}
\sqrt{\sum_{\scriptstyle i,j|(x_i,y_j)\in \Omega}\left(-\frac{4 \tilde u_{ij} - \tilde u_{i - 1, j} -\tilde  u_{i + 1, j} -\tilde  u_{i, j - 1} -\tilde  u_{i, j + 1}}{h^2}-f_{ij}\right)^2}\\
{\leqslant
\sqrt{
\dfrac{1}{h^2}\times \dfrac{h^4}{12^2} \|u\|_{\mathcal{C}^4(\Omega)}^2}
\leqslant Ch\|u\|_{\mathcal{C}^4(\Omega)}.    }
\end{multline}
Moreover, thanks to Lemma \ref{lemma2},
{
$$\sum_{\scriptstyle i,j|(x_i,y_j)\in \Omega}e_{ij}^2
=\dfrac{1}{h^2}\|e_h\|_{h,0,\Omega_h}^2
\leq C_P\left(
\dfrac{1}{h^2}\|e_h\|_{h,1,\Omega_h}^2
+hb_h(e_h,e_h)
\right)
\leq C|||e_h|||_h^2.
$$}
Thus, 
\begin{equation*}
    (I)\leq Ch\| u\|_{\mathcal{C}^4(\Omega)}|||e_h|||_h.
\end{equation*}

\textbf{Term (II):}
Consider $w:=\tilde u/\varphi$. 
Let $(x_{i},y_j)\in \partial\Omega_h$ such that $(x_{i+1},y_j)\in\Omega$.
Using Sobolev inequality and Hardy inequality (see e.g. \cite{duprez})
$$\|w\|_{\mathcal{C}^{1}([x_i,x_{i+1}])}\leqslant C \|w\|_{2,[x_i,x_{i+1}]}
\leqslant C \|\tilde u\|_{3,[x_i,x_{i+1}]}.
$$
Hence 
\begin{multline*}\left|\dfrac{\varphi_{(i+1)j}\tilde u_{i,j}-\varphi_{ij}\tilde u_{i+1,j}}{\sqrt{\varphi_{(i+1)j}^2+\varphi_{ij}^2}}\right|
\leq \left|\dfrac{\varphi_{(i+1)j}\varphi_{ij}}{\min\{|\varphi_{(i+1)j}|,|\varphi_{ij}|\}}\right||w(x_i,y_i)-w(x_{i+1},y_j)|\\
\leq \max\{|\varphi_{ij}|,|\varphi_{i+1,j}|\}|w(x_i,y_i)-w(x_{i+1},y_j)|\\
 \leq Ch\|\varphi\|_{L^{\infty}([x_i,x_{i+1}])}\|w\|_{\mathcal{C}^{1}([x_i,x_{i+1}])}
 \leq Ch^2\|\tilde u\|_{3,[x_i,x_{i+1}]}.
\end{multline*}
Thus, since the number of edges where is applied the ghost penalty is of order $\frac{CL}{h}$,
\begin{multline*}
(II)\leqslant b_h(\tilde U,\tilde U)^{1/2}b_h(e_h,e_h)^{1/2}\\
\leqslant \dfrac{C}{h}\left(\sqrt{\sum_{(i,j)\in B_x}\left|\dfrac{\varphi_{(i+1)j}\tilde u_{i,j}-\varphi_{ij}\tilde u_{i+1,j}}{\sqrt{\varphi_{(i+1)j}^2+\varphi_{ij}^2}}\right|^2}+\sqrt{\sum_{(i,j)\in B_y}\left|\dfrac{\varphi_{i(j+1)}\tilde u_{i,j}-\varphi_{ij}\tilde u_{i(j+1)}}{\sqrt{\varphi_{i(j+1)}^2+\varphi_{ij}^2}}\right|^2}\right)|||e_h|||_h\\
\leqslant C\sqrt{h}\|\tilde u\|_{3,\overline{\Omega}_h}|||e_h|||_h.
\end{multline*}

\textbf{Term (III):}
 Again, since the number of edges where the ghost penalty is applied is of order $\frac{CL}{h}$,
\begin{multline*}
\sum_{(i,j)\in J_x}\frac{-\tilde u_{i - 1, j} + 2 \tilde u_{ij} - \tilde u_{i + 1, j}}{h }
  \times \frac{- e_{i - 1, j} + 2 e_{ij} - e_{i + 1, j}}{h }\\
  \leqslant
 Ch\|\tilde u\|_{\mathcal{C}^2(\Omega_h)} \sum_{(i,j)\in J_x}\left| \frac{- e_{i - 1, j} + 2 e_{ij} - e_{i + 1, j}}{h }\right|\\
\leqslant Ch^{1/2}\|\tilde u\|_{\mathcal{C}^2(\Omega_h)} \left(\sum_{(i,j)\in J_x}\left| \frac{- e_{i - 1, j} + 2 e_{ij} - e_{i + 1, j}}{h }\right|^2\right)^{1/2}.
  \end{multline*}
Thus
$$(III)\leq Ch^{1/2}\|\tilde u\|_{\mathcal{C}^2(\overline{\Omega}_h)}|||e_h|||_h.$$
Combining with Lemma \ref{lemma2}, 
$$\|e_h\|_{h,1,\Omega}\leqslant h|||e_h|||_h \leqslant Ch^{3/2}\|u\|_{\mathcal{C}^4(\Omega)}.$$
Lemma \ref{lemma2} leads to the $L^{\infty}$ and $L^2$ estimates.
\end{proof}

Let us now prove Theorem \ref{theo:cond}. 
\begin{proof}[Proof of Theorem \ref{theo:cond}]
    
Thanks to Proposition \ref{prop:coer} and Lemma \ref{lemma2},
$$a_h(v_h,v_h)\geq C\sum_{(i,j):(x_i,y_j)\in \Omega_h}v_{ij}^2.$$
Moreover, thanks to the expression of $a_h$
$$a_h(v_h,v_h)\leq \dfrac{C}{h^2}\sum_{(i,j):(x_i,y_j)\in \Omega_h}v_{ij}^2,$$
which leads to Theorem \ref{theo:cond}. 

\end{proof}

\section{Alternative scheme}
\label{sec:alter}
Here, we propose an alternative version of the scheme that is more complex but (numerically) optimally convergent { for the $H^1$-norm}.

In 2D, consider the following finite difference scheme:
find a discrete function $u_h=(u_{ij})_{ij}$ defined on $\Omega_h$ such that
\begin{equation*}
\tilde a_h(u_h,v_h)=l_h(v_h),
\end{equation*}
for all discrete function $v_h=(v_{ij})_{ij}$ defined on $\Omega_h$,
where
$$\tilde a_h(u_h,v_h)=(-\Delta_h u_h,v_h)+\tilde b_h(u_h,v_h)+\tilde j_h(u_h,v_h),$$
with 
\begin{multline*}
\tilde b_h(u_h,v_h) =
\frac{\gamma}{2 h^2}
\left(\sum_{ij} \frac{u_{(i-1, j)-(i+1,j)}^{\varphi}\times v_{(i-1, j)-(i+1,j)}^{\varphi}}
{4\varphi_{i+1,j}^2\varphi_{i-1,j}^2+\varphi_{ij}^2\varphi_{i-1,j}^2+\varphi_{ij}^2\varphi_{i+1,j}^2} \right. \\
 \left. +\sum_{ij} \frac{u_{(i, j-1)-(i,j+1)}^{\varphi}\times v_{(i, j-1)-(i,j+1)}^{\varphi}}{4\varphi_{i,j+1}^2\varphi_{i,j-1}^2+\varphi_{ij}^2\varphi_{i,j-1}^2+\varphi_{ij}^2\varphi_{i,j+1}^2}\right)
  \end{multline*}
and
$$u_{(i-1, j)-(i+1,j)}^{\varphi}:=2\varphi_{i+1}\varphi_{i-1}u_i-\varphi_i\varphi_{i-1}u_{i+1}-\varphi_i\varphi_{i+1}u_{i-1},$$
$u_{(i, j-1)-(i,j+1)}^{\varphi}$ and $v_{(i, j-1)-(i,j+1)}^{\varphi}$ are similarly defined, and the second stabilization term is given by
\begin{multline}\label{eq:ghost3}
\tilde j_h(u_h,v_h)= \sigma\bigg( \sum_{i,j}\frac{- u_{i - 1, j} + 3 u_{ij} - 3u_{i + 1, j}+u_{i + 2, j}}{h }
  \times \frac{- v_{i - 1, j} + 3 v_{ij} - 3v_{i + 1, j}+v_{i + 2, j}}{h }
  \\
 +\sum_{i,j}\frac{- u_{i , j- 1} + 3 u_{ij} - 3u_{i , j+ 1}+u_{i, j + 2}}{h }
  \times \frac{- v_{i , j- 1} + 3 v_{ij} - 3v_{i, j + 1}+v_{i , j+ 2}}{h }
  \bigg).
\end{multline}
{The indices in the sums are such that all the corresponding nodes belong to $\Omega$ with one outside to $\Omega$.} 

{
\begin{remark}
    This alternative scheme is 
    given in the 2D case for readability but is still holding in 3D by adding the terms corresponding to the third index. We will give in Section \ref{sec:first num}  numerical illustrations in both cases.
We do not give convergence proof for this alternative scheme, but it can be analyzed in future work.
\end{remark}
}

Let us explain how to obtain the penalization term $\tilde b_h$. 
If we assume that $u=p\varphi$ with $p=p_0+p_1(x-x_i)$ and $u_{ij}=u(x_i,y_j)$, then
$$\begin{cases}
    u_{i+1,j}=(p_0+p_1h)\varphi_{i+1,j},\\
        u_{ij}=p_0\varphi_{ij},\\
            u_{i-1,j}=(p_0-p_1h)\varphi_{i-1,j},
\end{cases}$$
which gives
$$u_{(i-1, j)-(i+1,j)}^{\varphi}=0.$$

Concerning the stabilization term \eqref{eq:ghost3}, $\partial_xu(x_i,y_i)$ can be approximated (with an order 2) by
$$\dfrac{u(x_{i+1},y_i)-u(x_{i-1},y_i)}{2h} \text{ and }
\dfrac{-3u(x_{i},y_i)+4u(x_{i+1},y_i)-u(x_{i+2},y_i)}{2h},$$
which gives for the jump of  $\partial_xu(x_i,y_i)$
$$\frac{-u(x_{i+1},y_i)+3u(x_{i},y_i)-3u(x_{i+1},y_i)+u(x_{i+2},y_i)}{2h}.$$
Thus \eqref{eq:ghost3} is an approximation of \eqref{ghost}.

\section{Numerical illustrations}\label{sec:first num}

{In this section, we compare our two schemes with different existing approaches: 
\begin{itemize}
    \item $\varphi$-FEM scheme: to illustrate the interest of our new approach, it is mandatory to compare it numerically with $\varphi$-FEM \cite{duprez} to highlight the advantages and drawbacks of a finite element approach compared to a finite difference approach;
    \item a standard finite element method: we also compare our method to the {generic} 
    technique to solve PDEs, a classic conforming finite element method; 
    \item Shortley-Weller approach: We finally compare our method to the finite difference scheme of the literature. For that, we have implemented the Shortley-Weller method \cite{yoon,bramble}. The method has the same objective, which is to deal with complex geometries using a finite difference approach, but the associated matrix is not well conditioned. It is then natural to compare our work with this technique. 
\end{itemize}
{
The schemes presented in Section \ref{sec:main} and \ref{sec:alter} will be denoted in the different figures by $\varphi$-FD and $\varphi$-FD2, respectively.
The FEM schemes are written thanks to the \texttt{FEniCS} software (see \cite{logg2010dolfin}) and the finite difference schemes using the  \texttt{python} libraries {\texttt{scipy}\footnote{\url{https://scipy.org/}} \cite{scipy} and \texttt{numpy}\footnote{\url{https://numpy.org/}} \cite{numpy}}. } 
{The simulations were executed on a laptop with an \texttt{Intel Core i7-12700H CPU} and $32$Gb of memory. All the codes to reproduce the results are available at
\begin{center}
    \url{https://github.com/PhiFEM/PhiFD.git}
\end{center}}
}

Since the solution of $\varphi$-FD is defined only on the nodes $(x_i,y_j)_{ij}$ and the solutions to Shortley-Weller and Standard FEM live only on $\Omega$, then the $\varphi$-FEM and Standard FEM solutions will be interpolated on the nodes $(x_i,y_j)_{ij}$ belonging to $\Omega$.
The relative errors will then be computed 
thanks to the norms $\| \cdot \|_{h,0} $, $\| \cdot \|_{h,\infty} $ and $\| \cdot \|_{h,1} $ defined in Section \ref{sec:main}.

{Note that this way of calculating errors for finite element methods may slightly deteriorate the results compared to the standard way of calculating them. The idea is to compare the same quantities for each scheme.}

\subsection{First test case : 2D example}

We consider the explicit solution
\[ u = \cos \left( \frac{\pi}{2} r \right) \]
on the circle centered at $(0.5,0.5)$ with a radius $R=0.3+1e-10$
and $$r = \frac{1}{R}\sqrt{{(x-0.5)^2} + {(y-0.5)^2}}.$$ 
This choice of radius ensures that the real boundary cuts an edge close to a node. In this case, the Shortley-Weller approach will not be well-conditioned.

For the $\varphi$-FD scheme, the theoretical rate $h^{3/2}$ is reached for the $H^1$  norm and we observe a $h^2$ rate for the $L^2$ and $L^{\infty}$ norms (see Figure \ref{fig:error1}  and \ref{fig:error2}, left and Table \ref{tab:my_label}). 
$\varphi$-FD2 seems less good for coarse grids but is slightly better for fine resolution and has the optimal convergence $h^2$ in particular for the $H^1$ norm.
We also have the optimal conditioning number of the corresponding matrix with an order of $1/h^2$ (see Figure \ref{fig:error2}, right). The \texttt{python}  code has less than 100 lines (see Appendix) and uses only the libraries \texttt{scipy} and \texttt{numpy}, which induces a reduced computational time (see Figure \ref{fig:error-time}). 
{On these figures, it appears that $\varphi$-FEM and $\varphi$-FD both have interests to solve PDEs. Indeed, while the $L^2$ and $L^\infty$ are pretty close for the two approaches, the $H^1$ error, the conditioning or the computation times are much different: the $\varphi$-FD approach is much faster than the finite element approach while it leads to a slightly worst error on the derivatives of the solution.} Moreover, for the two $\varphi$-FD schemes, we observe the supraconvergence phenomenon as for the Shortley-Weller approach. { The theoretical and experimental convergence rate of the $H^1$-norm is larger than the FEM ones, $\mathcal{O}(h^{3/2})$ instead of $\mathcal{O}(h)$. 
This increase of accuracy is called supraconvergence. 
}

\definecolor{darkviolet}{rgb}{0.58, 0.0, 0.83}
\definecolor{cardinal}{rgb}{0.77, 0.12, 0.23}
\definecolor{coral}{rgb}{1.0, 0.5, 0.31}

\begin{figure}[!ht]
\centering
\begin{tikzpicture}
\begin{loglogaxis}[name = ax1, width = .45\textwidth, xlabel = $h$, 
            ylabel = $L^2$ relative error,
            legend style = { at={(1,1)},anchor=south east, legend columns =2,
			/tikz/column 2/.style={column sep = 10pt}}]
      \addplot[mark=x, darkviolet] coordinates {
(0.14142135623730964,0.031519579722260126)
(0.07071067811865482,0.007105909572318148)
(0.03535533905932741,0.001717523433752962)
(0.01767766952966378,0.00042734812624873166)
(0.00883883476483197,0.00010634395666944391)
(0.004419417382415985,2.6507986003439807e-05)
};
\addplot[mark=*] coordinates {
(0.10418890663488801,0.05227379978073314)
(0.05189897025587035,0.012798117171171332)
(0.02602192321559839,0.0032515180823360166)
(0.013030362894840943,0.0008276123625753482)
(0.006515355748437418,0.00020489914440907024)
(0.003257817559846729,5.1337513087827325e-05)
};
\addplot[mark=+,ForestGreen] coordinates {
(0.1414213562373095,0.014073329671113055)
(0.07071067811865475,0.004138091745804491)
(0.035355339059327376,0.000852146922668053)
(0.017677669529663688,0.00022159723604357236)
(0.008838834764831844,5.432846300624399e-05)
(0.004419417382415922,1.414228798096821e-05)
};
\addplot[mark=diamond*,cardinal] coordinates {
(0.1414213562373095,0.015330827686081596)
(0.07071067811865475,0.0033789089800612144)
(0.035355339059327376,0.0008620940986203928)
(0.017677669529663688,0.00020507051643194895)
(0.008838834764831844,5.0490510612585105e-05)
(0.004419417382415922,1.245055344631594e-05)
};
\addplot[mark=triangle*,coral] coordinates {
(0.1414213562373095,0.01785026427235157)
(0.07071067811865475,0.00516851481620919)
(0.035355339059327376,0.0013975683243563346)
(0.017677669529663688,0.00035804269225948144)
(0.008838834764831844,9.067163034266747e-05)
(0.004419417382415922,2.2809894499005287e-05)
};
\logLogSlopeTriangle{0.53}{0.2}{0.12}{2}{black};
\legend{$\varphi$-FEM, Std-FEM, SW, $\varphi$-FD, $\varphi$-FD2}
\end{loglogaxis}
\end{tikzpicture}
\quad
\begin{tikzpicture}
\begin{loglogaxis}[name = ax1, width = .45\textwidth, xlabel = $h$, 
            ylabel = $L^\infty$ relative error,
            legend style = {  at={(1,1)},anchor=south east, legend columns =2,
			/tikz/column 2/.style={column sep = 10pt}}]
      \addplot[mark=x, darkviolet] coordinates {
(0.14142135623730964,0.02098142362383006)
(0.07071067811865482,0.005152287666035191)
(0.03535533905932741,0.0013238675936899048)
(0.01767766952966378,0.00033503208163524313)
(0.00883883476483197,8.467243677297585e-05)
(0.004419417382415985,2.1336517889861777e-05)
};
\addplot[mark=*] coordinates {
(0.10418890663488801,0.05147358745601338)
(0.05189897025587035,0.012515284804823341)
(0.02602192321559839,0.0037446906690078007)
(0.013030362894840943,0.0009469495438065326)
(0.006515355748437418,0.00021984941882140983)
(0.003257817559846729,6.398194618869712e-05)
};
\addplot[mark=+,ForestGreen] coordinates {
(0.1414213562373095,0.015604867237412948)
(0.07071067811865475,0.005264253394800217)
(0.035355339059327376,0.001185005438104144)
(0.017677669529663688,0.0003019129406690106)
(0.008838834764831844,7.875902141956521e-05)
(0.004419417382415922,2.0345753299360074e-05)
};
\addplot[mark=diamond*,cardinal] coordinates {
(0.1414213562373095,0.017476665445005167)
(0.07071067811865475,0.0037540923580809643)
(0.035355339059327376,0.0009598193402651961)
(0.017677669529663688,0.00024399986729963923)
(0.008838834764831844,6.98692133466733e-05)
(0.004419417382415922,1.8783553609440675e-05)
};
\addplot[mark=triangle*,coral] coordinates {
(0.1414213562373095,0.019511924779577063)
(0.07071067811865475,0.00539182306234819)
(0.035355339059327376,0.001414203734315635)
(0.017677669529663688,0.00035900300377722397)
(0.008838834764831844,9.049251220416615e-05)
(0.004419417382415922,2.2714104650961843e-05)
};
\logLogSlopeTriangle{0.53}{0.2}{0.12}{2}{black};
\legend{$\varphi$-FEM, Std-FEM, SW, $\varphi$-FD, $\varphi$-FD2}
\end{loglogaxis}
\end{tikzpicture}
\caption{\textbf{First test case, a 2D example.} $L^{2}$ (left) and $L^{\infty}$ (right) relative errors with respect to the discretization step for $\varphi$-FEM, standard FEM, Shortley-Weller, $\varphi$-FD and $\varphi$-FD2. 
}\label{fig:error1} 
 \end{figure}
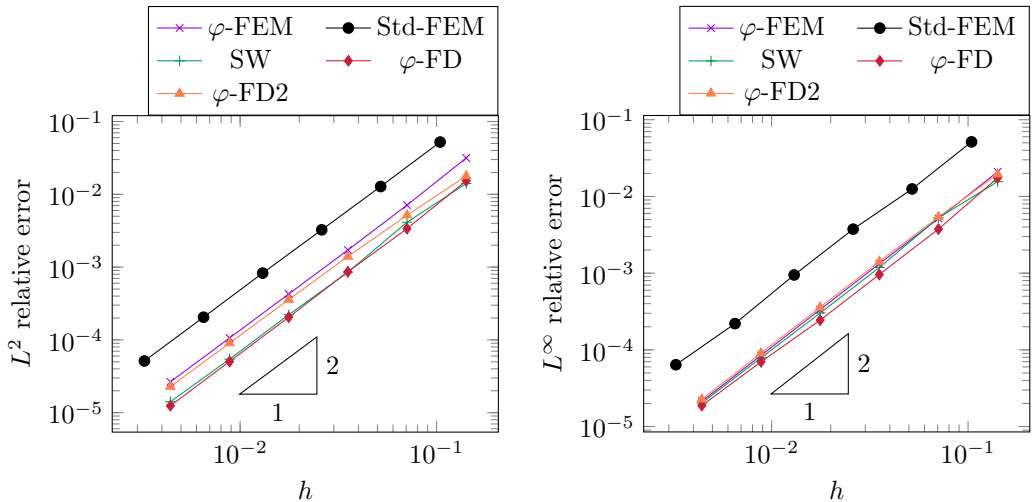 

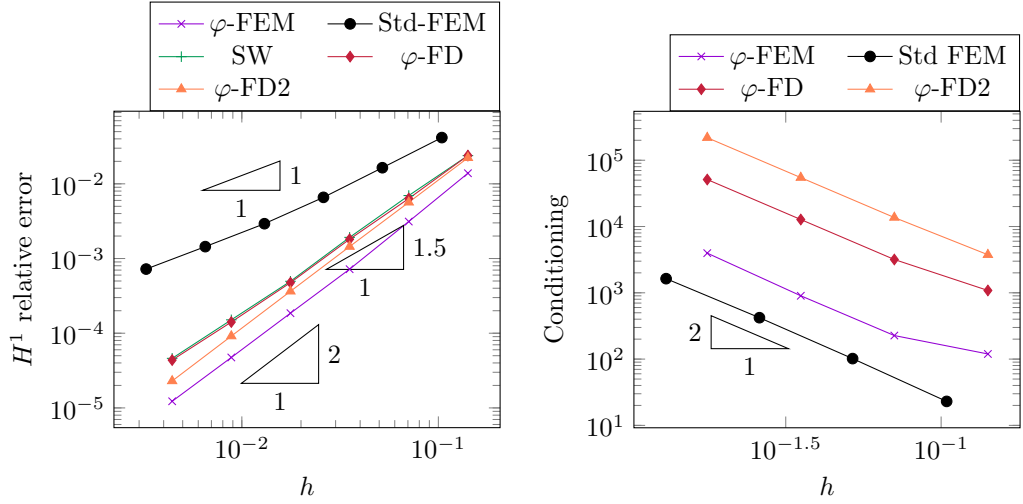
\begin{figure}[!ht]
\centering
\begin{tikzpicture}
\begin{loglogaxis}[name = ax1, width = .45\textwidth, xlabel = $h$, 
            ylabel = $H^1$ relative error,
            legend style = {  at={(1,1)},anchor=south east, legend columns =2,
			/tikz/column 2/.style={column sep = 10pt}}]
      \addplot[mark=x, darkviolet] coordinates {
(0.14142135623730964,0.01386339809396037)
(0.07071067811865482,0.003125765570139466)
(0.03535533905932741,0.0007199505532968982)
(0.01767766952966378,0.00018550080930124924)
(0.00883883476483197,4.73449200793911e-05)
(0.004419417382415985,1.2266193393025097e-05)
};
\addplot[mark=*] coordinates {
(0.10418890663488801,0.04160250730047812)
(0.05189897025587035,0.01642936336778085)
(0.02602192321559839,0.006591166648129152)
(0.013030362894840943,0.002924082997394163)
(0.006515355748437418,0.001442649597173449)
(0.003257817559846729,0.0007223924521806728)
};
\addplot[mark=+,ForestGreen] coordinates {
(0.1414213562373095,0.023868958142536614)
(0.07071067811865475,0.006968673181499712)
(0.035355339059327376,0.0019191819833754816)
(0.017677669529663688,0.0004956450048084294)
(0.008838834764831844,0.00015100630039573633)
(0.004419417382415922,4.607759065068968e-05)
};
\addplot[mark=diamond*,cardinal] coordinates {
(0.1414213562373095,0.023885806391539526)
(0.07071067811865475,0.0063544283924225385)
(0.035355339059327376,0.0018129444281420196)
(0.017677669529663688,0.00047517138175493944)
(0.008838834764831844,0.0001410857936298765)
(0.004419417382415922,4.368289547275989e-05)
};
\addplot[mark=triangle*,coral] coordinates {
(0.1414213562373095,0.022280314854866607)
(0.07071067811865475,0.005652972187080991)
(0.035355339059327376,0.001444541293698455)
(0.017677669529663688,0.0003641171069311043)
(0.008838834764831844,9.148682835786856e-05)
(0.004419417382415922,2.2929945539583716e-05)
};
\logLogSlopeTriangle{0.53}{0.2}{0.14}{2}{black};
\logLogSlopeTriangle{0.75}{0.2}{0.5}{1.5}{black};
\logLogSlopeTriangle{0.43}{0.2}{0.75}{1}{black};
\legend{$\varphi$-FEM, Std-FEM,SW, $\varphi$-FD, $\varphi$-FD2}
\end{loglogaxis}
\end{tikzpicture}
\quad
\begin{tikzpicture}
\begin{loglogaxis}[name = ax1, width = .45\textwidth, xlabel = $h$,   ylabel = Conditioning,
            legend style = { at={(1,1)},anchor=south east, legend columns =2,
			/tikz/column 2/.style={column sep = 10pt}}]
    \addplot[mark=x, darkviolet] coordinates {
    (0.14142135623730964,119.49456333487727)
    (0.07071067811865482,226.56093224668254)
    (0.03535533905932741,901.8110719970891)
    (0.01767766952966378,3966.245842574167)
      };
    \addplot[mark=*] coordinates {
    (0.10418890663488801,22.997824008133747)
    (0.05189897025587035,101.72332149961157)
    (0.02602192321559839,421.22129287440896)
    (0.013030362894840943,1634.0591286161348)      };
    \addplot[mark=diamond*,cardinal] coordinates {
      (0.1414213562373095,1084.9136669607178)
      (0.07071067811865475,3180.567742275954)
      (0.035355339059327376,12747.558351349535)
      (0.017677669529663688,51142.24151659187)
      };
    \addplot[mark=triangle*,coral] coordinates {
      (0.1414213562373095,3752.1456697096455)
      (0.07071067811865475,13670.810197750205)
      (0.035355339059327376,54683.63872911465)
      (0.017677669529663688,218735.16126612716)
      };
      
\logLogSlopeTriangleinv{0.4}{0.2}{0.25}{2}{black};
\legend{$\varphi$-FEM, Std FEM, $\varphi$-FD,$\varphi$-FD2}
\end{loglogaxis}
\end{tikzpicture}

\caption{\textbf{First test case, a 2D example.}  $H^1$ relative error (left) and conditioning number (right) with respect to the discretization step for $\varphi$-FEM, standard FEM, Shortley-Weller, $\varphi$-FD and $\varphi$-FD2.}\label{fig:error2} 
 \end{figure} 

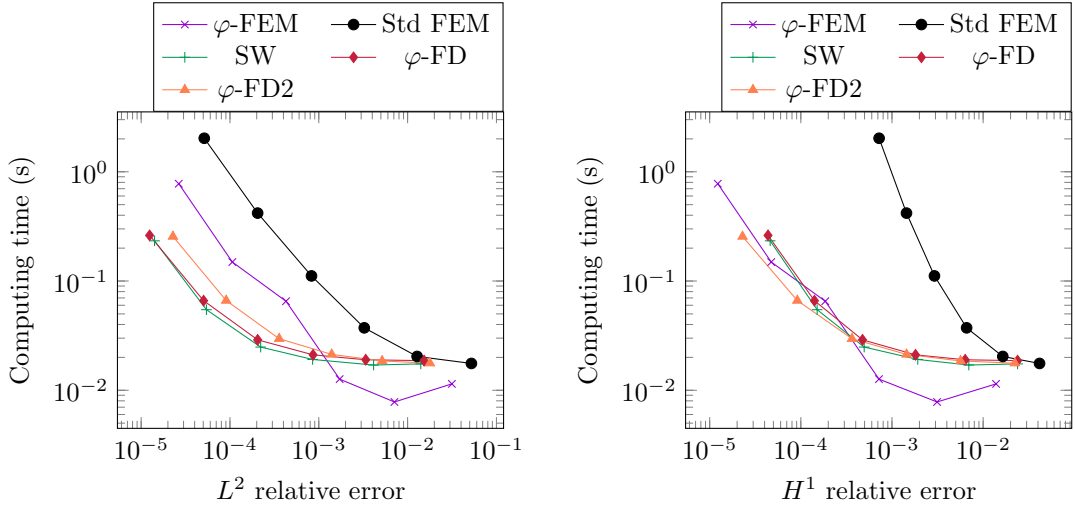
\begin{figure}[!ht]
\centering
\begin{tikzpicture}
\begin{loglogaxis}[width = .45\textwidth, ylabel = Computing time (s), name = ax2, at = {(ax1.south east)}, xshift = 2cm,
            xlabel =$L^2$ relative error, 
            legend style = {  at={(1,1)},anchor=south east, legend columns =2,
			/tikz/column 2/.style={column sep = 10pt}}, 
                ]
           \addplot[mark=x, darkviolet] coordinates {
(0.031519579722260126,0.011440753936767578)
(0.007105909572318148,0.007812738418579102)
(0.001717523433752962,0.01265096664428711)
(0.00042734812624873166,0.06533527374267578)
(0.00010634395666944391,0.1491687297821045)
(2.6507986003439807e-05,0.7775013446807861)
};
\addplot[mark=*] coordinates {
(0.05227379978073314,0.01760244369506836)
(0.012798117171171332,0.02039813995361328)
(0.0032515180823360166,0.03728294372558594)
(0.0008276123625753482,0.11136674880981445)
(0.00020489914440907024,0.4184756278991699)
(5.1337513087827325e-05,2.026933431625366)
};
\addplot[mark=+,ForestGreen] coordinates {
(0.014073329671113055,0.017437219619750977)
(0.004138091745804491,0.017021656036376953)
(0.000852146922668053,0.01914358139038086)
(0.00022159723604357236,0.024857282638549805)
(5.432846300624399e-05,0.054651737213134766)
(1.414228798096821e-05,0.2333214282989502)
};
\addplot[mark=diamond*,cardinal] coordinates {
(0.015330827686081596,0.018668651580810547)
(0.0033789089800612144,0.019017934799194336)
(0.0008620940986203928,0.021099567413330078)
(0.00020507051643194895,0.028948068618774414)
(5.0490510612585105e-05,0.06607651710510254)
(1.245055344631594e-05,0.2617220878601074)
};
\addplot[mark=triangle*,coral] coordinates {
(0.01785026427235157,0.01769733428955078)
(0.00516851481620919,0.01856088638305664)
(0.0013975683243563346,0.021288633346557617)
(0.00035804269225948144,0.029649734497070312)
(9.067163034266747e-05,0.06628227233886719)
(2.2809894499005287e-05,0.25555896759033203)
};        
\legend{$\varphi$-FEM, Std FEM,SW, $\varphi$-FD, $\varphi$-FD2}
\end{loglogaxis}
\end{tikzpicture}
\quad
\begin{tikzpicture}
\begin{loglogaxis}[width = .45\textwidth, ylabel = Computing time (s), name = ax2, at = {(ax1.south east)}, xshift = 2cm,
            xlabel =$H^1$ relative error, 
            legend style = {  at={(1,1)},anchor=south east, legend columns =2,
			/tikz/column 2/.style={column sep = 10pt}}, 
                ]
      \addplot[mark=x, darkviolet] coordinates {
(0.01386339809396037,0.011440753936767578)
(0.003125765570139466,0.007812738418579102)
(0.0007199505532968982,0.01265096664428711)
(0.00018550080930124924,0.06533527374267578)
(4.73449200793911e-05,0.1491687297821045)
(1.2266193393025097e-05,0.7775013446807861)
};
\addplot[mark=*] coordinates {
(0.04160250730047812,0.01760244369506836)
(0.01642936336778085,0.02039813995361328)
(0.006591166648129152,0.03728294372558594)
(0.002924082997394163,0.11136674880981445)
(0.001442649597173449,0.4184756278991699)
(0.0007223924521806728,2.026933431625366)
};
\addplot[mark=+,ForestGreen] coordinates {
(0.023868958142536614,0.017437219619750977)
(0.006968673181499712,0.017021656036376953)
(0.0019191819833754816,0.01914358139038086)
(0.0004956450048084294,0.024857282638549805)
(0.00015100630039573633,0.054651737213134766)
(4.607759065068968e-05,0.2333214282989502)
};
\addplot[mark=diamond*,cardinal] coordinates {
(0.023885806391539526,0.018668651580810547)
(0.0063544283924225385,0.019017934799194336)
(0.0018129444281420196,0.021099567413330078)
(0.00047517138175493944,0.028948068618774414)
(0.0001410857936298765,0.06607651710510254)
(4.368289547275989e-05,0.2617220878601074)
};
\addplot[mark=triangle*,coral] coordinates {
(0.022280314854866607,0.01769733428955078)
(0.005652972187080991,0.01856088638305664)
(0.001444541293698455,0.021288633346557617)
(0.0003641171069311043,0.029649734497070312)
(9.148682835786856e-05,0.06628227233886719)
(2.2929945539583716e-05,0.25555896759033203)
};
\legend{$\varphi$-FEM, Std FEM,SW, $\varphi$-FD, $\varphi$-FD2}
\end{loglogaxis}
\end{tikzpicture}
\caption{\textbf{First test case, a 2D example.} Computing times with respect to the $L^2$ relative error (Left) and the $H^1$ relative error (Right) for $\varphi$-FEM, standard FEM, Shortley-Weller, $\varphi$-FD and $\varphi$-FD2.}\label{fig:error-time} 
\end{figure} 

\begin{table}[!ht]
    \centering
\begin{tabular}{|c|c|c|c|c|c|}
&$\varphi$-FEM&Std FEM&SW&$\varphi$-FD&$\varphi$-FD2\\
Relative $L^2$-error   & 2.04&2.0&2.01&2.05&1.93\\
Relative $L^{\infty}$-error&1.98&1.94&1.95&1.96&1.95\\
Relative $H^1$-error&2.02&1.17&1.82&1.83&1.98\\
\end{tabular}
    \caption{\textbf{First test case, a 2D example.} Orders of convergence.}
    \label{tab:my_label}
\end{table}

To complete this test case and to justify our choice for the parameters $\sigma$ and $\gamma$, we present in Figure \ref{fig:stability_parameters} 
the evolution of the $L^2$ relative error and the condition number of the matrix. This leads to the choice of $\sigma=0.01$ for both schemes and $\gamma= 1$ for the first $\varphi$-FD scheme and $\gamma=10$ for the second scheme. 
{We remark in Fig.~\ref{fig:stability_parameters} that the $L^2$ relative error of the second $\varphi$-FD scheme is more stable to the variations of $\sigma$ than the one of $\varphi$-FD, thanks to the second order term  $\tilde{j}_h$.}

\begin{figure}
    \centering
\begin{tikzpicture}\begin{loglogaxis}[name = ax1, width = .45\textwidth, ylabel = $L^2$ relative error,
xlabel = $\sigma$,
legend style = {at={(1,1)},anchor=south east, legend columns =2,
 /tikz/column 2/.style={column sep = 10pt},
  /tikz/column 4/.style={column sep = 10pt},
  /tikz/column 6/.style={column sep = 10pt}}]
\addplot[mark=*, darkviolet] coordinates {
(0.001,0.013482556811757468)
(0.01,0.015330827686081596)
(0.1,0.03446637668032492)
(1.0,0.08288744622595297)
(5.0,0.0996545440700707)
(10.0,0.1023672676132408)
(20.0,0.10379367126687414)
};

\addplot[mark=*, darkviolet, densely dashed] coordinates {
(0.001,0.019047896099894417)
(0.01,0.01785026427235157)
(0.1,0.012907721706709953)
(1.0,0.015184002570503006)
(5.0,0.01696200126057095)
(10.0,0.017238475545525517)
(20.0,0.017381943238118952)
};

\addplot[mark=x, cardinal] coordinates {
(0.001,0.003157619306759415)
(0.01,0.0033789089800612144)
(0.1,0.006262857619567839)
(1.0,0.015241660800261235)
(5.0,0.018831741036814106)
(10.0,0.019443502449753836)
(20.0,0.01976926621474449)
};

\addplot[mark=x, cardinal, densely dashed] coordinates {
(0.001,0.005328719280811668)
(0.01,0.00516851481620919)
(0.1,0.00449730446005262)
(1.0,0.004176083512989966)
(5.0,0.004208588884415343)
(10.0,0.0042173810098919025)
(20.0,0.004222357874225764)
};

\addplot[mark=+, ForestGreen] coordinates {
(0.001,0.0008218059433541441)
(0.01,0.0008620940986203928)
(0.1,0.0013198691357178011)
(1.0,0.0031183221379327062)
(5.0,0.004026554482389484)
(10.0,0.0042000259509069795)
(20.0,0.0042958425370915035)
};

\addplot[mark=+, ForestGreen, densely dashed] coordinates {
(0.001,0.0014104775532671606)
(0.01,0.0013975683243563344)
(0.1,0.0013271443705287606)
(1.0,0.0012481009505046691)
(5.0,0.0012360522080803117)
(10.0,0.0012349970593544316)
(20.0,0.0012345848819652406)
};

\addplot[mark=triangle, coral] coordinates {
(0.001,0.000199155169239603)
(0.01,0.00020507051643194895)
(0.1,0.0002892200356736449)
(1.0,0.0006307390800359072)
(5.0,0.000790772295741882)
(10.0,0.000820111812843949)
(20.0,0.0008363369522984611)
};

\addplot[mark=triangle, coral, densely dashed] coordinates {
(0.001,0.00035999812831802005)
(0.01,0.0003580426922594815)
(0.1,0.000347705962979847)
(1.0,0.00033541649802375847)
(5.0,0.0003328668542449432)
(10.0,0.0003325218030533616)
(20.0,0.00033234976595158375)
};

\end{loglogaxis}
\begin{loglogaxis}[name = ax2, at={(ax1.south east)}, xshift=2cm, width = .45\textwidth, ylabel = Condition number,
xlabel = $\sigma$,
legend style = {at={(1,1.02)},anchor=south east, legend columns =4,
 /tikz/column 2/.style={column sep = 10pt},
  /tikz/column 4/.style={column sep = 10pt},
  /tikz/column 6/.style={column sep = 10pt}}]
\addplot[mark=*, darkviolet] coordinates {
(0.001,10685.496973981239)
(0.01,1084.9136669607178)
(0.1,871.7114261164152)
(1.0,1807.00952004618)
(5.0,6264.93613945238)
(10.0,11868.873029263726)
(20.0,23083.470314350696)
};

\addlegendentry{$\varphi$-FD $h=$0.14}
\addplot[mark=*, darkviolet, densely dashed] coordinates {
(0.001,36543.79419058322)
(0.01,3752.1456697096455)
(0.1,3508.210376670436)
(1.0,5612.167577278479)
(5.0,21283.76455761581)
(10.0,41159.82027244204)
(20.0,80960.17512200115)
};

\addlegendentry{$\varphi$-FD2 $h=$0.14}
\addplot[mark=x, cardinal] coordinates {
(0.001,10687.087104630795)
(0.01,3180.567742275954)
(0.1,3432.0873058964357)
(1.0,7205.680893276665)
(5.0,25659.417130986476)
(10.0,48861.34477484918)
(20.0,95295.76794832177)
};

\addlegendentry{$\varphi$-FD $h=$0.07}
\addplot[mark=x, cardinal, densely dashed] coordinates {
(0.001,36499.30165600433)
(0.01,13670.810197750205)
(0.1,14032.693522769006)
(1.0,22745.556521432336)
(5.0,84370.47057473166)
(10.0,162174.66156401008)
(20.0,317931.95593083516)
};

\addlegendentry{$\varphi$-FD2 $h=$0.07}
\addplot[mark=+, ForestGreen] coordinates {
(0.001,12743.483542640904)
(0.01,12747.558351349535)
(0.1,13692.236068619897)
(1.0,28214.69263957704)
(5.0,103442.00998916924)
(10.0,197861.68650417766)
(20.0,386775.25521482195)
};

\addlegendentry{$\varphi$-FD $h=$0.04}
\addplot[mark=+, ForestGreen, densely dashed] coordinates {
(0.001,54543.6664004808)
(0.01,54683.63872911465)
(0.1,56131.31510858904)
(1.0,81780.69580935218)
(5.0,330636.492671842)
(10.0,646068.8548820791)
(20.0,1277449.6355356607)
};

\addlegendentry{$\varphi$-FD2 $h=$0.04}
\addplot[mark=triangle, coral] coordinates {
(0.001,51140.16734725085)
(0.01,51142.24151659187)
(0.1,54771.108305898924)
(1.0,114289.61949737286)
(5.0,416944.3379786287)
(10.0,797651.5357303553)
(20.0,1559442.0092255436)
};

\addlegendentry{$\varphi$-FD $h=$0.02}
\addplot[mark=triangle, coral, densely dashed] coordinates {
(0.001,218175.24273622842)
(0.01,218735.16126612716)
(0.1,224526.14941712614)
(1.0,347829.5648022929)
(5.0,1336765.5677214288)
(10.0,2604045.47379523)
(20.0,5144113.277784081)
};

\addlegendentry{$\varphi$-FD2 $h=$0.02}
\end{loglogaxis}
\begin{loglogaxis}[name = ax3, width = .45\textwidth,yshift=-5.25cm, ylabel = $L^2$ relative error,
xlabel = $\gamma$,
legend style = {at={(1,1)},anchor=south east, legend columns =2,
 /tikz/column 2/.style={column sep = 10pt},
  /tikz/column 4/.style={column sep = 10pt},
  /tikz/column 6/.style={column sep = 10pt}}]
\addplot[mark=*, darkviolet] coordinates {
(0.001,5.972958751238064)
(0.01,0.5388136900128185)
(0.1,0.052631852214158246)
(1.0,0.015330827686081596)
(10.0,0.01436023099310335)
(20.0,0.014378283653213437)
};

\addplot[mark=*, darkviolet, densely dashed] coordinates {
(0.001,0.17755861578667473)
(0.01,0.3246285071748899)
(0.1,0.0977580005949103)
(1.0,0.022975331270932216)
(10.0,0.01785026427235157)
(20.0,0.01860302905696651)
};

\addplot[mark=x, cardinal] coordinates {
(0.001,0.784391648040964)
(0.01,0.07603412558344884)
(0.1,0.008320245126328165)
(1.0,0.0033789089800612144)
(10.0,0.003264093132763761)
(20.0,0.0032724484992747425)
};

\addplot[mark=x, cardinal, densely dashed] coordinates {
(0.001,1.306443972856408)
(0.01,0.07190490373056234)
(0.1,0.006457980030494001)
(1.0,0.0043786971385551396)
(10.0,0.00516851481620919)
(20.0,0.0052415233609218875)
};

\addplot[mark=+, ForestGreen] coordinates {
(0.001,0.1043662931983969)
(0.01,0.010796382919287919)
(0.1,0.0015096390189973123)
(1.0,0.0008620940986203928)
(10.0,0.0008190351856074286)
(20.0,0.0008096939265729985)
};

\addplot[mark=+, ForestGreen, densely dashed] coordinates {
(0.001,0.04162231924943216)
(0.01,0.003428929133730719)
(0.1,0.0009904681311119117)
(1.0,0.0013288821120134814)
(10.0,0.0013975683243563344)
(20.0,0.0014025165350792252)
};

\addplot[mark=triangle, coral] coordinates {
(0.001,0.023321262648542375)
(0.01,0.002343756309128225)
(0.1,0.00033155883005805975)
(1.0,0.00020507051643194895)
(10.0,0.00019818330983537266)
(20.0,0.00019689375938356527)
};

\addplot[mark=triangle, coral, densely dashed] coordinates {
(0.001,0.016130920355590998)
(0.01,0.0005216236253383578)
(0.1,0.0002876843343624397)
(1.0,0.0003490616721453896)
(10.0,0.0003580426922594815)
(20.0,0.0003586269102795196)
};

\end{loglogaxis}
\begin{loglogaxis}[name = ax4, at={(ax3.south east)}, xshift=2cm, width = .45\textwidth, ylabel = Condition number,
xlabel = $\gamma$,
legend style = {at={(1,1.02)},anchor=south east, legend columns =4,
 /tikz/column 2/.style={column sep = 10pt},
  /tikz/column 4/.style={column sep = 10pt},
  /tikz/column 6/.style={column sep = 10pt}}]
\addplot[mark=*, darkviolet] coordinates {
(0.001,53642.53510656272)
(0.01,4217.540545721521)
(0.1,1463.0597946900036)
(1.0,1084.9136669607178)
(10.0,3612.7992268630533)
(20.0,6611.539998802812)
};

\addplot[mark=*, darkviolet, densely dashed] coordinates {
(0.001,38079.64319290825)
(0.01,11150.44609299109)
(0.1,1865.2188469122061)
(1.0,1278.096660539153)
(10.0,3752.1456697096455)
(20.0,6813.850389845219)
};

\addplot[mark=x, cardinal] coordinates {
(0.001,24689.829563322903)
(0.01,3260.3794378044404)
(0.1,3148.4048739296395)
(1.0,3180.567742275954)
(10.0,13665.414101292137)
(20.0,25649.440676407878)
};

\addplot[mark=x, cardinal, densely dashed] coordinates {
(0.001,625503.9888428771)
(0.01,8759.688408668932)
(0.1,3180.975606054707)
(1.0,3319.2780130831256)
(10.0,13670.810197750205)
(20.0,25654.104869707637)
};

\addplot[mark=+, ForestGreen] coordinates {
(0.001,20353.34272933246)
(0.01,12740.76535844634)
(0.1,12740.930586425724)
(1.0,12747.558351349535)
(10.0,54661.834770280635)
(20.0,102597.82665388453)
};

\addplot[mark=+, ForestGreen, densely dashed] coordinates {
(0.001,67783.31354662703)
(0.01,12767.15840342791)
(0.1,12769.550944250746)
(1.0,12878.893289688349)
(10.0,54683.63872911465)
(20.0,102616.57871679051)
};

\addplot[mark=triangle, coral] coordinates {
(0.001,51138.52804221878)
(0.01,51138.53946733974)
(0.1,51138.66039104109)
(1.0,51142.24151659187)
(10.0,218647.6176391304)
(20.0,410391.3954943867)
};

\addplot[mark=triangle, coral, densely dashed] coordinates {
(0.001,875404.5254571759)
(0.01,51156.76837633022)
(0.1,51159.752494982684)
(1.0,51948.237443105834)
(10.0,218735.16126612716)
(20.0,410466.5241067579)
};

\end{loglogaxis}
\end{tikzpicture}
\caption{\textbf{First test case, a 2D example.} 
Top: evolution of the $L^2$ relative errors (left) and condition number (right) with respect to $\sigma$ with $\gamma=1$ for $\varphi$-FD and $\gamma=10$ for $\varphi$-FD2. Bottom: evolution of the $L^2$ relative errors (left) and condition number (right) with respect to $\gamma$ with $\sigma=0.01$ for $\varphi$-FD and $\varphi$-FD2.}\label{fig:stability_parameters}
\end{figure}
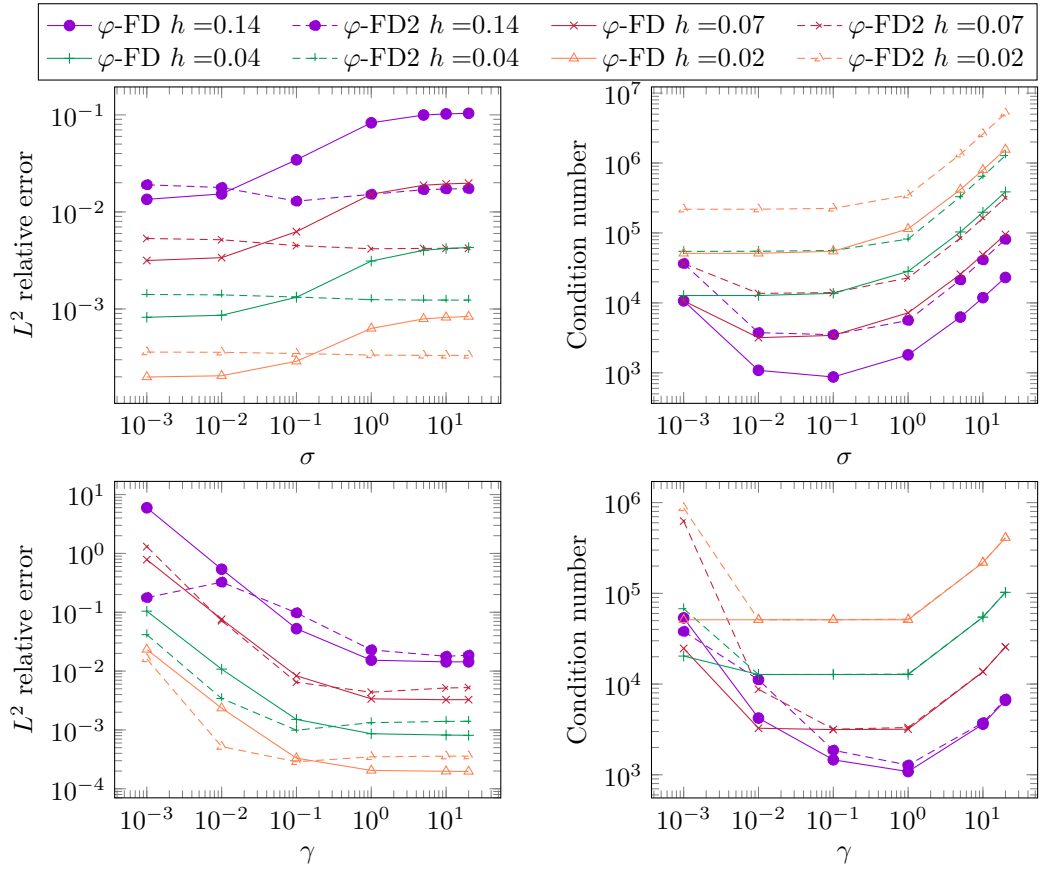

\subsection{Second test case: a 3D example}
We now consider a three-dimensional extension of the previous test case, i.e., the same explicit solution, in a sphere centered at $(0.5,0.5,0.5)$, with a radius $R = 0.3$ and 
\[r = \frac{1}{R}\sqrt{(x-0.5)^2 + (y-0.5)^2 + (z-0.5)^2}\,. \]

Once again, the optimal $h^2$ convergence is observed in the $L^2$ and $H^1$ norms (see Fig.~\ref{fig:error_L2_3D}). {Moreover, our two schemes outperform the two finite element methods {as well as the Shortley-Weller approach}.}

\begin{figure}[!ht]
  \centering
  \begin{tikzpicture}
  \begin{loglogaxis}[name = ax1, width = .45\textwidth, xlabel = $h$, 
              ylabel = $L^2$ relative error,
              legend style = { at={(1,1)},anchor=south east, legend columns =2,
        /tikz/column 2/.style={column sep = 10pt}}]
     \addplot[mark=x, darkviolet] coordinates {
(0.13323467750529835,0.08341681598786607)
(0.10188534162169868,0.04314611413929906)
(0.08247860988423233,0.02618574342570295)
(0.06928203230275516,0.01810880258595267)
(0.0597258899161683,0.013205458672567862)
(0.052486388108147944,0.010077366566985001)
(0.042245141648021435,0.0063946169550997455)
};
\addplot[mark=*] coordinates {
(0.14605582938901915,0.06281077629014724)
(0.09952415050160912,0.037236232103842)
(0.07442191545542995,0.023938555344681933)
(0.059298268538494875,0.01796360435331951)
(0.04965607769491114,0.015018882809393204)
(0.042349252328999984,0.013084332858446106)
(0.0373073627526282,0.01193781358038744)
};
\addplot[mark=+,ForestGreen] coordinates {
(0.11785113019775793,0.017500292657541175)
(0.08838834764831845,0.011929873036535294)
(0.07071067811865475,0.00577596266603567)
(0.05892556509887897,0.0037019247203674793)
(0.05050762722761054,0.0030810402686364657)
(0.04419417382415922,0.002388085713820419)
(0.035355339059327376,0.0014123937056703313)
};
\addplot[mark=diamond*,cardinal] coordinates {
(0.11785113019775793,0.009753717778436305)
(0.08838834764831845,0.00555470119731129)
(0.07071067811865475,0.0035008173261974244)
(0.05892556509887897,0.0024436046001389544)
(0.05050762722761054,0.0017823500612321247)
(0.04419417382415922,0.0014018801911592136)
(0.035355339059327376,0.000869556551363951)
};
\addplot[mark=triangle*,coral] coordinates {
(0.11785113019775793,0.014202962719482267)
(0.08838834764831845,0.008133570782494035)
(0.07071067811865475,0.005487971964463171)
(0.05892556509887897,0.0038856762359612503)
(0.05050762722761054,0.0028655392487522656)
(0.04419417382415922,0.0022134065336017602)
(0.035355339059327376,0.0014322396452467923)
};
    
  \logLogSlopeTriangle{0.53}{0.2}{0.12}{2}{black};
  \legend{$\varphi$-FEM, Std-FEM,SW,  $\varphi$-FD, $\varphi$-FD2}
  \end{loglogaxis}
  \end{tikzpicture}
  \quad
  \begin{tikzpicture}
  \begin{loglogaxis}[name = ax1, width = .45\textwidth, xlabel = $h$, 
              ylabel = $H^1$ relative error,
              legend style = {  at={(1,1)},anchor=south east, legend columns =2,
        /tikz/column 2/.style={column sep = 10pt}}]
\addplot[mark=x, darkviolet] coordinates {
(0.13323467750529835,0.0419045876238796)
(0.10188534162169868,0.019827323321024114)
(0.08247860988423233,0.011481016625787087)
(0.06928203230275516,0.007664687843199209)
(0.0597258899161683,0.005581982251160239)
(0.052486388108147944,0.0042299461228534375)
(0.042245141648021435,0.002637861861265995)
};
\addplot[mark=*] coordinates {
(0.14605582938901915,0.06169695127733516)
(0.09952415050160912,0.03506415877159911)
(0.07442191545542995,0.025232582815480055)
(0.059298268538494875,0.019334025365728858)
(0.04965607769491114,0.014665322955418043)
(0.042349252328999984,0.012685342456948459)
(0.0373073627526282,0.011784490354973443)
};
\addplot[mark=+,ForestGreen] coordinates {
(0.11785113019775793,0.01976369577996817)
(0.08838834764831845,0.009852564567696138)
(0.07071067811865475,0.007555917346743244)
(0.05892556509887897,0.004854261680718782)
(0.05050762722761054,0.004108911706457663)
(0.04419417382415922,0.002849872830639567)
(0.035355339059327376,0.0019279882516062155)
};
\addplot[mark=diamond*,cardinal] coordinates {
(0.11785113019775793,0.01893755399635284)
(0.08838834764831845,0.01020485364498072)
(0.07071067811865475,0.007002994025556655)
(0.05892556509887897,0.004767459201157143)
(0.05050762722761054,0.0037889199665225933)
(0.04419417382415922,0.002990968746838273)
(0.035355339059327376,0.0018923183390207932)
};
\addplot[mark=triangle*,coral] coordinates {
(0.11785113019775793,0.016691819378303)
(0.08838834764831845,0.009250584226112099)
(0.07071067811865475,0.005946832554891426)
(0.05892556509887897,0.00414427156175043)
(0.05050762722761054,0.0030381140935776757)
(0.04419417382415922,0.0023340612499293473)
(0.035355339059327376,0.0014981038585128027)
};

  \logLogSlopeTriangle{0.53}{0.2}{0.12}{2}{black};
  \legend{$\varphi$-FEM, Std-FEM, SW, $\varphi$-FD, $\varphi$-FD2}
  \end{loglogaxis}
  \end{tikzpicture}
  \caption{\textbf{Second test case, a 3D example.} $L^{2}$ (left) and $H^1$ (right) relative errors with respect to the discretisation step for $\varphi$-FEM, standard FEM, Shortley-Weller, $\varphi$-FD and $\varphi$-FD2.
  }\label{fig:error_L2_3D} 
   \end{figure}
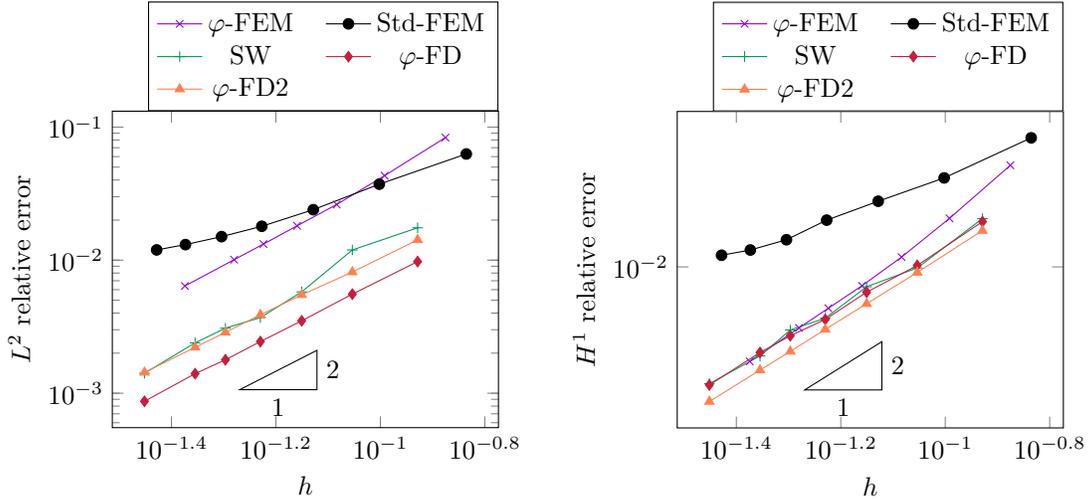

\subsection{Third test case: combination with a multigrid approach}\label{sec:multi}

Another advantage of using Cartesian grids is that we can take advantage of the structured multigrid solvers \cite{adams} in order to improve the stability and to speed up the numerical method. The multigrid method is based on combining relaxation schemes and a hierarchy of coarser grids. After applying point relaxation on the finest grid, a correction term is found by representing the fine-grid residual on the next coarsest grid and using point relaxation there. Recursively, a hierarchy of grids is obtained, and the algorithm is stopped when the problem is coarsened enough to be solved directly.  \cite{fedorenko} describes the different iterative techniques for solving elliptic difference problems: simple iteration method, Seidel's method, Richardson's method, Young's method, relaxation method and minimal residuals method.  \cite{hackbush} gives a description of a multigrid method for the solution of Poisson equation on general bounded regions with numerical examples. Two important components in multigrid methods are the restriction and prolongation operators, which transfer the information between grids. In \cite{Ruggiu}, they have used Summation-by-Parts preserving interpolation operators, which lead to accurate and stable coarse-grid approximations. In the last section of the present article, we propose a multigrid-like technique to obtain a good compromise regarding the computation time concerning the error.

To reduce the computational time of the numerical resolution, we propose a way to combine 
our numerical scheme $\varphi$-FD with a multigrid approach. 
The idea is to use the $\varphi$-FD solution obtained on a coarse grid using a direct linear solver to initialize the $\varphi$-FD resolution on a finer grid in the case of an iterative resolution of the associated linear system. More precisely, the algorithm will be divided into three steps: 
\begin{enumerate}
    \item\textbf{Step 1, direct resolution on coarse grid:} we compute a coarse $\varphi$-FD solution $u_0$ on a coarse grid 
    {$N_0^n$} with a direct solver.
    \item \textbf{Step 2, interpolation on the fine grid:} we consider $u_1$ the interpolation by splines {of order 2} of $u_0$ on a given fine grid 
    {$N_{\text{obj}}^n$} with $N_{\text{obj}}>> N_0$.
    \item \textbf{Step 3, iterative resolution on fine grid:} we compute a $\varphi$-FD solution $u_2$ on the fine grid with an iterative linear solver and $u_1$ as initialisation.
\end{enumerate}
{In 2D, we will compare this algorithm with the two following methods:}
\begin{itemize}
  \item \textbf{Direct method:} we solve a problem with a direct solver for several resolutions $N_0\times N_0$ and we interpolate the solution to the fine grid 
  {$N_{\text{obj}}^n$}. {The direct solver used here is the standard one from \texttt{scipy}, i.e. a LU solver.} 
  \item \textbf{Iterative method:} the same process is applied except that we use this time an iterative solver{, namely the stabilized conjugate bigradient}. 
\end{itemize}
{In 3D, we only compare our approach to the iterative method.}
We will consider the 2D and 3D examples presented in the previous subsections.  $N_{\text{obj}}$ will be fixed to $2200$ and $200$ for the 2D and 3D cases, respectively.  
{All the iterative solvers have the same tolerance for the interior relative residues, i.e. $10^{-4}$.}
All the compatible iterative solvers of the python library \texttt{scipy} have been tested by the authors, but the stabilized conjugate bigradient\footnote{https://docs.scipy.org/doc/scipy/reference/generated/scipy.sparse.linalg.bicgstab.html} has always proven to be the best. Note that the simple conjugate gradient cannot be used because the matrix $A$ is not symmetric. 
\begin{remark}
\begin{itemize}
    \item Another point is that one can also add an intermediate step, solving a finer problem with resolution $N_0 < N_1 < N_{\text{obj}}$ to reduce the number of iterations of the last solver. However, this approach was not necessary for our test cases and increased the number of parameters to tune in the pipeline (tolerance and maximal number of iterations of the intermediate solver, intermediate grid, parameters of the intermediate interpolation). 
    \item If a $\varphi$-FD scheme is subsequently developed for non-linear equations, this multigrid approach can be applied to the iterations of Newton's algorithm.
\end{itemize}
\end{remark}

The results in Fig.~\ref{fig:results_multigrid} (left) illustrate that our approach is better than the two baseline methods: indeed, we reach a better precision (due to the final iterative solver) much faster since we only need a few iterations of the fine linear solver. On each baseline curve, we add the discretization used for the resolution, and on the multigrid curves, the one used for the coarse solvers. 
{Since we have chosen to use the multigrid approach using an interpolation of $f$ and $\varphi$ from the fine resolution to the coarse one, the computation times contain only the times to solve the linear systems and the time to interpolate $u$ from the resolution $N_0$ to the $N_{\text{obj}}$ for the multigrid approach.}

As previously said, one of the issues of the $\varphi$-FD technique, and all the finite difference techniques is the growth of the size of the linear system to solve, especially in 3D: the matrix $A$ collects $(N+1)^6$ values for a resolution $N$. Hence, one would always need to use an iterative solver to solve 3D problems with this approach. However, applying an iterative solver without any initial guess with a resolution $N=200$ leads to solving a problem with a matrix $A$ collecting more than $10^{13}$ values. Even using the sparsity of the matrix results in a gigantic system that takes a long time to solve. As illustrated in Fig.~\ref{fig:results_multigrid} (right), our approach gives results to such problems much faster than the baseline method, {the iterative method presented before.}

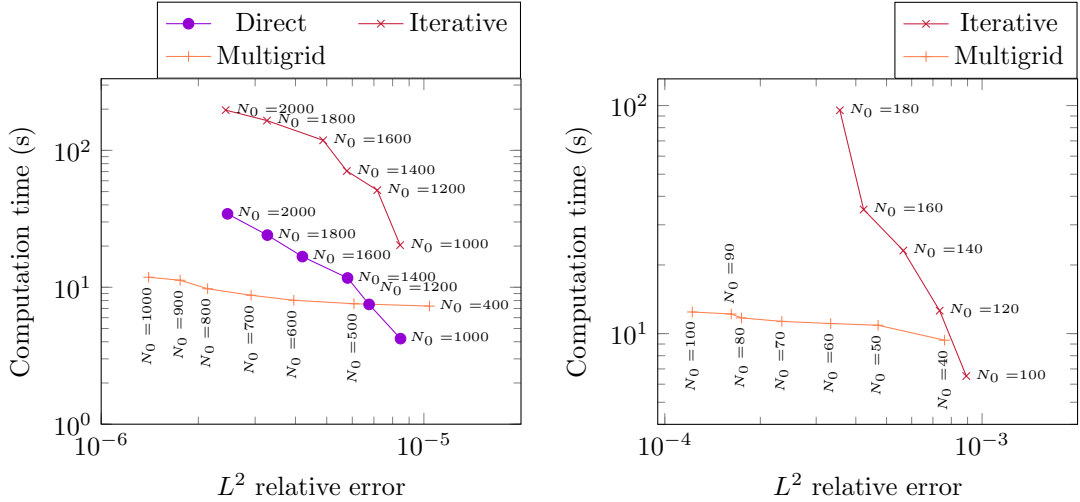
\begin{figure}
    \centering
\begin{tikzpicture}\begin{loglogaxis}[name = ax1, width = .48\textwidth, ylabel = Computation time (s),
xlabel = $L^2$ relative error,
xmin = 0.000001, xmax = 0.00002, ymin = 1,
legend style = {at={(1,1)},anchor=south east, legend columns =2,
 /tikz/column 2/.style={column sep = 10pt}}]
\addplot[mark=*, darkviolet] coordinates {
(8.46495990033093e-06,4.215105772018433)
(6.765245075799401e-06,7.4975221157073975)
(5.802659837899504e-06,11.69192624092102)
(4.20335684361316e-06,16.750582456588745)
(3.2729222256836493e-06,24.046744108200073)
(2.4635933515668634e-06,34.37465715408325)
};

\node [right] at (axis cs:  8.46495990033093e-06,  4.215105772018433){\tiny$N_0=$1000};
\node [above right] at (axis cs:  6.765245075799401e-06,  7.4975221157073975){\tiny$N_0=$1200};
\node [right] at (axis cs:  5.802659837899504e-06,  11.69192624092102){\tiny$N_0=$1400};
\node [right] at (axis cs:  4.20335684361316e-06,  16.750582456588745){\tiny$N_0=$1600};
\node [right] at (axis cs:  3.2729222256836493e-06,  24.046744108200073){\tiny$N_0=$1800};
\node [right] at (axis cs:  2.4635933515668634e-06,  34.37465715408325){\tiny$N_0=$2000};
\addplot[mark=x, cardinal] coordinates {
(8.441198578731359e-06,20.31014108657837)
(7.169975254426898e-06,51.19670391082764)
(5.7708410521783124e-06,70.68300437927246)
(4.872962067997632e-06,118.49665594100952)
(3.264392388318373e-06,165.0491168498993)
(2.431201015473319e-06,196.8179292678833)
};

\node [right] at (axis cs:  8.441198578731359e-06,  20.31014108657837){\tiny$N_0=$1000};
\node [right] at (axis cs:  7.169975254426898e-06,  51.19670391082764){\tiny$N_0=$1200};
\node [right] at (axis cs:  5.7708410521783124e-06,  70.68300437927246){\tiny$N_0=$1400};
\node [right] at (axis cs:  4.872962067997632e-06,  118.49665594100952){\tiny$N_0=$1600};
\node [right] at (axis cs:  3.264392388318373e-06,  165.0491168498993){\tiny$N_0=$1800};
\node [right] at (axis cs:  2.431201015473319e-06,  196.8179292678833){\tiny$N_0=$2000};
\addplot[mark=+, coral] coordinates {
(1.0410069705403135e-05,7.288126707077026)
(6.07177044339354e-06,7.5857861042022705)
(3.952906774566685e-06,8.032482624053955)
(2.914809986440399e-06,8.745643138885498)
(2.1366247210256463e-06,9.75935435295105)
(1.7567809641359133e-06,11.256510496139526)
(1.4028155817868523e-06,11.841386556625366)
};

\node [right] at (axis cs:  1.0410069705403135e-05,  7.288126707077026){\tiny$N_0=$400};
\node [left, rotate=90] at (axis cs:  6.07177044339354e-06,  7.5857861042022705){\tiny$N_0=$500};
\node [left, rotate=90] at (axis cs:  3.952906774566685e-06,  8.032482624053955){\tiny$N_0=$600};
\node [left, rotate=90] at (axis cs:  2.914809986440399e-06,  8.745643138885498){\tiny$N_0=$700};
\node [left, rotate=90] at (axis cs:  2.1366247210256463e-06,  9.75935435295105){\tiny$N_0=$800};
\node [left, rotate=90] at (axis cs:  1.7567809641359133e-06,  11.256510496139526){\tiny$N_0=$900};
\node [left, rotate=90] at (axis cs:  1.4028155817868523e-06,  11.841386556625366){\tiny$N_0=$1000};
\legend{Direct, Iterative, Multigrid} 
\end{loglogaxis}
\end{tikzpicture}
\quad 
\begin{tikzpicture}\begin{loglogaxis}[name = ax1, width = .48\textwidth, ylabel = Computation time (s),
xlabel = $L^2$ relative error,
xmin = 0.000095, xmax = 0.002, ymin=4,
legend style = {at={(1,1)},anchor=south east, legend columns =1,
 /tikz/column 2/.style={column sep = 10pt}}]
\addplot[mark=x, cardinal] coordinates {
(0.0008920134924437446,6.513899087905884)
(0.0007350803579282396,12.604397535324097)
(0.0005647123515437393,23.131813287734985)
(0.0004233780979161378,35.00911092758179)
(0.00035625871265390304,95.35810470581055)
};

\node [right] at (axis cs:  0.0008920134924437446,  6.513899087905884){\tiny$N_0=$100};
\node [right] at (axis cs:  0.0007350803579282396,  12.604397535324097){\tiny$N_0=$120};
\node [right] at (axis cs:  0.0005647123515437393,  23.131813287734985){\tiny$N_0=$140};
\node [right] at (axis cs:  0.0004233780979161378,  35.00911092758179){\tiny$N_0=$160};
\node [right] at (axis cs:  0.00035625871265390304,  95.35810470581055){\tiny$N_0=$180};
\addplot[mark=+, coral] coordinates {
(0.0007615090577193,9.35623812675476)
(0.000470216361471,10.879122734069824)
(0.0003333142977443,11.071059226989746)
(0.0002338150300595,11.311166048049929)
(0.0001744681262475,11.71899938583374)
(0.00016196463218,12.189433813095093)
(0.0001220917695514,12.439875364303589)
};

\node [left, rotate=90] at (axis cs:  0.0007615090577193,  9.35623812675476){\tiny$N_0=$40};
\node [left, rotate=90] at (axis cs:  0.000470216361471,  10.879122734069824){\tiny$N_0=$50};
\node [left, rotate=90] at (axis cs:  0.0003333142977443,  11.071059226989746){\tiny$N_0=$60};
\node [left, rotate=90] at (axis cs:  0.0002338150300595,  11.311166048049929){\tiny$N_0=$70};
\node [left, rotate=90] at (axis cs:  0.0001744681262475,  11.71899938583374){\tiny$N_0=$80};
\node [right, rotate=90] at (axis cs:  0.00016196463218,  12.189433813095093){\tiny$N_0=$90};
\node [left, rotate=90] at (axis cs:  0.0001220917695514,  12.439875364303589){\tiny$N_0=$100};
\legend{Iterative, Multigrid} 
\end{loglogaxis}
\end{tikzpicture}
  \caption{{\textbf{Third test case, multigrid approach:} Computational time with respect to the $L^2$ relative error for the direct, iterative and multigrid method for 2D (left) and 3D (right) examples.} 
  }\label{fig:results_multigrid}
\end{figure}

\section{Conclusion and perspectives}

In this work, we have proposed a well-conditioned finite difference method inspired by the $\varphi$-FEM approach for solving elliptic PDEs on general geometries. The key advantages of the proposed $\varphi$-FD method can be summarized as follows:

\begin{itemize}
    \item \textbf{Well-conditioned Matrices:} The method produces well-conditioned matrices, which ensure stability and efficiency during the numerical resolution of PDEs.
    \item \textbf{Quasi-optimal Convergence:} The $\varphi$-FD scheme achieves quasi-optimal convergence rates, demonstrating accuracy comparable to other established methods.
    \item \textbf{Compatibility with Multigrid Techniques:} Our method is fully compatible with multigrid approaches, allowing further acceleration of the numerical solution process, especially for large-scale problems.
 \end{itemize}
The proposed method opens several avenues for future research and development:
\begin{itemize}
    \item \textbf{Neumann Boundary Conditions:} An extension of the $\varphi$-FD method to handle Neumann boundary conditions is a natural next step, enabling the application of this technique to a broader class of PDEs.
    \item \textbf{Proof for the Second Scheme:} While we have introduced an alternative $\varphi$-FD scheme, proof of its convergence properties is still pending. This will be an essential step to validate and potentially optimize the scheme further.
    \item \textbf{Non-linear Problems and Multigrid Implementation:} Another promising direction is to apply the $\varphi$-FD scheme to non-linear PDEs, combined with a multigrid approach within Newton's iterative method. This could significantly enhance the efficiency and applicability of the method in solving complex, real-world problems.
    \item {\textbf{Combination with a neural network:} As it has been proposed in \cite{duprez2024varphi}, where $\varphi$-FEM is combined with a neural operator, one could also imagine an adaptation to the $\varphi$-FD approach to generate a collection of precise data to train a neural operator.
    }
    \item {\textbf{Supraconvergence:} as highlighted in Remark \ref{rem:convergence}, some supraconvergence phenomena appear in theoretical and numerical results. It could be interesting to analyze these aspects in detail. However, this behavior is probably due to the use of the classical Laplacian operator and would disappear in the case of more complex differential operators.}
    \item {\textbf{Regularity of $\varphi$ and polygonal geometries:} in the theoretical setting of our method, $\varphi$ is required to be $\mathcal{C}^2$ as in the $\varphi$-FEM framework. 
    Concerning polygonal geometries with outgoing corners, we can ensure enough regularity on $\varphi$. 
    Indeed, for instance, for the square $[0,1]^2$, we can take $\varphi=-x(x-1)y(y-1)$. This has been tested in the case of $\varphi$-FEM in \cite{duprez} and we can expect the same behavior for $\varphi$-FD. We remark that this definition of $\varphi$ is not strictly a level-set. 
    The case of entering corners seems more complex and deserves to be studied in greater detail. 
    }
\end{itemize}

The results obtained in this study indicate that the $\varphi$-FD method has significant potential in numerical analysis and computational science, particularly for problems involving complex geometries and large-scale computations.
However, the theoretical results of the present paper needs more regularity of the exact solution than the finite element approaches.

\section*{Acknowledgement}

The authors thank the anonymous referee whose comments improved the quality of the manuscript.

\appendix
\section{Example of code for \texorpdfstring{$\varphi$}{test}-FD in \texttt{python}}

\begin{lstlisting}[language=Python,backgroundcolor=\color{lightgray!30},
 caption=$\varphi$-FD Python implementation.]
import numpy as np
import scipy.sparse as sp
from scipy.sparse.linalg import spsolve

# Radius of the domain
R = 0.3 + 1e-10

# Parameter of penalization and stabilization
sigma, gamma = 0.01, 1.0

# Construction of the grid
Nx, Ny = 100, 100
x, y = np.linspace(0, 1, Nx + 1), np.linspace(0, 1, Ny + 1)
hx, hy = x[1] - x[0], y[1] - y[0]
X, Y = np.meshgrid(x, y)

# Computation of the exact solution, exact source term and the level-set
r = lambda x, y: np.sqrt((x - 0.5) * (x - 0.5) + (y - 0.5) * (y - 0.5) + 1e-12)
K = np.pi / 2 / R
ue = lambda x, y: np.cos(K * r(x, y))
f = lambda x, y: K * K * np.cos(K * r(x, y)) + K * np.sin(K * r(x, y)) / r(x, y)
phi = lambda x, y: (x - 0.5) * (x - 0.5) + (y - 0.5) * (y - 0.5) - R * R
phiij = phi(X, Y)
ind = (phiij < 0) + 0
mask = sp.diags(diagonals=ind.ravel())
indOut = 1 - ind

# Laplacian matrix
D2x = (1 / hx / hx) * sp.diags(
    diagonals=[-1, 2, -1], offsets=[-1, 0, 1], shape=(Nx + 1, Nx + 1)
)
D2y = (1 / hy / hy) * sp.diags(
    diagonals=[-1, 2, -1], offsets=[-1, 0, 1], shape=(Ny + 1, Ny + 1)
)
D2x_2d = sp.kron(sp.eye(Ny + 1), D2x)
D2y_2d = sp.kron(D2y, sp.eye(Nx + 1))
A = mask @ (D2x_2d + D2y_2d)

# Boundary conditions
diag = np.zeros((Nx + 1) * (Ny + 1))
diagxp = np.zeros((Nx + 1) * (Ny + 1) - 1)
diagxm = np.zeros((Nx + 1) * (Ny + 1) - 1)
diagyp = np.zeros((Nx + 1) * Ny)
diagym = np.zeros((Nx + 1) * Ny)
actGx = np.zeros((Ny + 1, Nx + 1))
actGy = np.zeros((Ny + 1, Nx + 1))

indx = ind[:, 1 : Nx + 1] - ind[:, 0:Nx]
J, I = np.where((indx == 1) | (indx == -1))
for k in range(np.shape(I)[0]):
    if indx[J[k], I[k]] == 1:
        indOut[J[k], I[k]], actGx[J[k], I[k] + 1] = 0, 1
    else:
        indOut[J[k], I[k] + 1], actGx[J[k], I[k]] = 0, 1
phiS = np.square(phiij[J, I]) + np.square(phiij[J, I + 1])
diag[I + (Nx + 1) * J] = phiij[J, I + 1] * phiij[J, I + 1] / phiS
diagxp[I + (Nx + 1) * J] = -phiij[J, I] * phiij[J, I + 1] / phiS
diag[I + 1 + (Nx + 1) * J] = phiij[J, I] * phiij[J, I] / phiS
diagxm[I + (Nx + 1) * J] = -phiij[J, I] * phiij[J, I + 1] / phiS

indy = ind[1 : Ny + 1, :] - ind[0:Ny, :]
J, I = np.where((indy == 1) | (indy == -1))
for k in range(np.shape(I)[0]):
    if indy[J[k], I[k]] == 1:
        indOut[J[k], I[k]], actGy[J[k] + 1, I[k]] = 0, 1
    else:
        indOut[J[k] + 1, I[k]], actGy[J[k], I[k]] = 0, 1
phiS = np.square(phiij[J, I]) + np.square(phiij[J + 1, I])
diag[I + (Nx + 1) * J] += phiij[J + 1, I] * phiij[J + 1, I] / phiS
diagyp[I + (Nx + 1) * J] = -phiij[J, I] * phiij[J + 1, I] / phiS
diag[I + (Nx + 1) * (J + 1)] += phiij[J, I] * phiij[J, I] / phiS
diagym[I + (Nx + 1) * J] = -phiij[J, I] * phiij[J + 1, I] / phiS

B = (gamma / hx / hy) * sp.diags(
    diagonals=(diagym, diagxm, diag, diagxp, diagyp),
    offsets=(-Nx - 1, -1, 0, 1, Nx + 1),
)

# Stabilization
maskGx = sp.diags(diagonals=actGx.ravel())
maskGy = sp.diags(diagonals=actGy.ravel())
C = sigma * hx * hy * (D2x_2d.T @ maskGx @ D2x_2d + D2y_2d.T @ maskGy @ D2y_2d)

# Penalization outside
D = sp.diags(diagonals=indOut.ravel())

# Linear system
A, b = (A + B + C + D).tocsr(), (ind * f(X, Y)).ravel()
u = spsolve(A, b).reshape(Ny + 1, Nx + 1)

# Computation of the errors
uref = ue(X, Y)
e = ind * (u - uref)
eL2 = np.linalg.norm(e) * np.sqrt(hx * hy)
emax = np.linalg.norm(e, np.inf)
print(eL2, emax)
\end{lstlisting}


\end{document}